\documentclass[reqno]{amsart}
\usepackage{enumerate}
\usepackage{color}
\usepackage{mathrsfs}
\usepackage{hyperref}
\usepackage{float}
\usepackage{graphicx}%
\hypersetup{
colorlinks,
citecolor=blue,
filecolor=black,
linkcolor=blue,
urlcolor=magenta
}
\hypersetup{linktocpage}

\usepackage{scrextend}
\changefontsizes[16pt]{11pt}

\newtheorem{theorem}{Theorem}[section]
\newtheorem{lemma}[theorem]{Lemma}

\theoremstyle{definition}

\newtheorem{remark}[theorem]{Remark}

\newtheorem{assumption}[theorem]{Assumption}

\newcommand{\norm}[1]{\left\Vert#1\right\Vert}

\numberwithin{equation}{section}
\usepackage{amsfonts}

\usepackage{amsmath}
\usepackage{amssymb}

\begin{document}
\font\nho=cmr10
\def\dive{\mathrm{div}}
\def\cal{\mathcal}
\def\L{\cal L}

\def \ud{\underline }
\def\id{{\indent }}
\def\f{\frac}
\def\non{{\noindent}}
 \def\le{\leqslant} 
 \def\leq{\leqslant}
 \def\geq{\geqslant} 
\def\rar{\rightarrow}
\def\Rar{\Rightarrow}
\def\ti{\times}
\def\i{\mathbb I}
\def\j{\mathbb J}
\def\si{\sigma}
\def\Ga{\Gamma}
\def\ga{\gamma}
\def\ld{{\lambda}}
\def\Si{\Psi}
\def\f{\mathbf F}
\def\r{\hro{R}}
\def\e{\cal{E}}
\def\B{\cal B}
\def\A{\mathcal{A}}
\def\p{\mathbb P}

\def\tet{\theta}
\def\Tet{\Theta}
\def\hro{\mathbb}
\def\ho{\mathcal}
\def\P{\ho P}
\def\E{\mathcal{E}}
\def\n{\mathbb{N}}
\def\M{\mathbb{M}}
\def\dMu{\mathbf{U}}
\def\dMcs{\mathbf{C}}
\def\dMcu{\mathbf{C^u}}
\def\vk{\vskip 0.2cm}
\def\td{\Leftrightarrow}
\def\df{\frac}
\def\Wei{\mathrm{We}}
\def\Rey{\mathrm{Re}}
\def\s{\mathbb S}
\def\l{\mathcal{L}}
\def\C+{C_+([t_0,\infty))}
\def\o{\cal O}

\begin{center}
{\LARGE\bf Well-posedness and exponential stability for Boussinesq systems on real hyperbolic Manifolds and application}

{\bf Pham Truong Xuan*}\footnote{*Corresponding author, Thang Long Institute of Mathematics and Applied Sciences (TIMAS), Thang Long University,
Nghiem Xuan Yem, Hoang Mai, Hanoi, Vietnam. Email: phamtruongxuan.k5@gmail.com or xuanpt@thanglong.edu.vn}
and
{\bf Tran Thi Ngoc}\footnote{Faculty of Fundamental Sciences, East Asia University of Technology,
Trinh Van Bo Street, Nam Tu Liem, Hanoi, Vietnam and Thang Long Institute of Mathematics and Applied Sciences (TIMAS), Thang Long University,
Nghiem Xuan Yem, Hoang Mai, Hanoi, Vietnam.
Email: ngoctt@eaut.edu.vn}

\end{center}

{\bf Abstract.} 
We investigate the global existence and exponential decay of mild solutions for the Boussinesq systems in $L^p$-phase spaces on the framework of real hyperbolic manifold $\mathbb{H}^d(\mathbb{R})$, where $d \geqslant 2$ and $1<p\leq d$. We consider a couple of Ebin-Marsden's Laplace and Laplace-Beltrami operators associated with the corresponding linear system which provides a vectorial matrix semigoup. First, we show the existence and the uniqueness of the bounded mild solution for the linear system by using dispersive and smoothing estimates of the vectorial matrix semigroup. Next, using the fixed point arguments, we can pass from the linear system to the semilinear system to establish the existence of the bounded mild solutions. By using Gronwall's inequality, we establish the exponential stability of such solutions. Finally, we give an application of stability to the existence of periodic mild solutions for the Boussinesq systems. 




\tableofcontents

\section{Introduction}
In the present paper, we are concerned with the incompressible Boussinesq system in the hyperbolic space $(\mathbb{H}^d(\mathbb{R}),g)$, where the dimension $d\geqslant 2$ and $g$ is the hyperbolic metric
\begin{equation}
\left\{
\begin{array}
[c]{rll}%
u_{t}+(u\cdot\nabla)u - Lu+\nabla p \!\!\!\! & = \kappa\theta h + \dive F\quad &
x\in\mathbb{H}^{d}(\r),\,t>0,\hfill\\
\operatorname{div}u\!\! & =\;0\quad & x\in\mathbb{H}^{d}(\r),\,t\geq0,\\
\theta_{t}-\widetilde{L}\theta+(u\cdot\nabla)\theta\!\! & =\;\dive f\quad & x\in
\mathbb{H}^{d}(\r),\,t>0,\\
u(x,0)\!\! & =\;u_{0}(x)\quad & x\in\mathbb{H}^{d}(\r),\\
\theta(x,0)\!\! & =\;\theta_{0}(x)\quad & x\in\mathbb{H}^{d}(\r),
\end{array}
\right.  \label{BouE}%
\end{equation}
where where $L= -(d-1)+ \overrightarrow{\Delta}$ is Ebin-Marsden's Laplace operator, $\widetilde{L}= \Delta_g$ is Laplace-Beltrami operator associated with metric $g$, the constant $\kappa>0$ is the volume expansion coefficient. The field $h$ is a generalized function of gravitational field satisfying Assumption \ref{Assum} below, and the constant $\kappa>0$ is the volume expansion coefficient. The unknowns $u$ is the velocity field, $p$ is the scalar pressure, and $\theta$ is the temperature. The vector field $f$ is given such that $\dive f$ represents the reference temperature and the second order tensor $F$ is given such that $\dive F$ represents the external force. Considering the zero-temperature case, i.e., $\theta=0$, then system (\ref{BouE}) becomes the Navier-Stokes equations.

We now reall briefly some results on the Boussinesq system in Euclidean space $\mathbb{R}^d$. Fife and Joseph \cite{Fi1969} provided one of the first rigorous mathematical results for the
convection problem by constructing analytic stationary solutions for the Boussinesq system with the bounded field $h$, as well as analyzing some stability and
bifurcation properties. After, Cannon and DiBenedetto \cite{Ca1980}
established the local-in-time existence in the class $L^{p}(0,T;L^{q}%
(\mathbb{R}^{n}))$ with suitable $p,q$. Hishida \cite{Hi1995} (see also \cite{Mo1991}) obtained the existence and exponential stability of global-in-time strong solutions for the Boussinesq system near to the steady state in a bounded domain of $\mathbb{R}^{3}$. Later, by using the $L^{p,\infty}$-$L^{q,\infty}$-dispersive and smoothing estimates in weak-$L^{p}$ spaces of the semigroup $e^{-tL}$ associated with the corresponding linear equations of the Boussinesq system,
Hishida \cite{Hi1997} showed the existence and large-time behavior of
global-in-time strong solutions in an exterior domain of $\mathbb{R}^{3}$
under smallness assumptions on the initial data $(u_{0},\theta_{0})$.
Well-posedness of time-periodic and almost periodic small solutions in exterior domains were proved
in \cite{HuyXuan22, Na2020} by employing frameworks based on weak-$L^{p}$ spaces. The
existence and stability of global small mild solutions for the Boussinesq system were studied in weak-$L^{p}$ spaces in \cite{Fe2006,Fe2010} and in Morrey spaces in \cite{Al2011}. A  result of stability in $B_{2,1}^{3/2}\times \dot{B}_{2,1}^{-1/2},$ under small perturbations, for a class of global large $H^{1}$- solutions was proved by \cite{Liu2014}. Brandolese and Schonbek \cite{Br2012} obtained results on the existence and time-decay of weak solutions for the Boussinesq system in whole space $\mathbb{R}^{3}$ with initial data $(u_{0},\theta_{0})\in L^{2}\times L^{2}$. Li and Wang \cite{Li-Wang2021}
analyzed the Boussinesq system in the torus $\mathbb{T}^{3}$ and obtained an
ill-posedness result in $\dot{B}_{\infty,\infty}^{-1}\times\dot{B}%
_{\infty,\infty}^{-1}$ by showing the so-called norm inflation phenomena. Komo
\cite{Komo2015} analyzed the Boussinesq system in general smooth domains
$\Omega\subset$ $\mathbb{R}^{3}$ and obtained uniqueness criteria for strong
solutions in the framework of Lebesgue time-spatial mixed spaces
$L^{p}(0,T;L^{q}(\Omega))$ by assuming $(u_{0},\theta_{0})\in L^{2}\times
L^{2}$ and $g\in L^{8/3}(0,T;L^{4}(\Omega))$. Considering the case of a
constant field $h$, Brandolese and He \cite{Br2020} showed the uniqueness of
mild solutions in the class $(u,\theta)\in C([0,T],L^{3}(\mathbb{R}^{3})\times
L^{1}(\mathbb{R}^{3}))$ with $\theta\in L_{loc}^{\infty}((0,T);L^{q,\infty
}(\mathbb{R}^{3}))$. The existence and uniqueness results in the partial inviscid cases of the Boussinesq system were studied in \cite{Danchin2009,Danchin2008}, where the authors
explored different kinds of conditions on the initial data $(u_{0},\theta
_{0})$ involving $L^{p},$ $L^{p,\infty}$ (weak-$L^{p}$) and Besov spaces.
Recently, the unconditional uniqueness of mild solutions for Boussinesq systems in Morrey-Lorentz spaces has established by Ferreira and Xuan \cite{FePha2023}. Additionally, The well-posedness and stability of periodic mild solutions for Boussinesq systems in weak-Morrey spaces has studied by Xuan et al. \cite{XuanVanThuy2024}.  

We present in the following some related works which concerne the Navier-Stokes equations and generalized evolution equations on non-compact manifolds with negative Ricci curvatures. On these manifolds, Ebin-Marsden \cite{EbiMa} introduced the notion of vectorial laplace operator by the mean of deformation tensor formula (today, it is known as Ebin and Marsden's laplace operator), then they reformulated the Navier-Stokes equations on Einstein manifolds that have negative Ricci curvatures. Since then, this notion has been used in the works of Czubak and Chan \cite{Cz1,Cz2} and also Lichtenfelz \cite{Li2016} to prove the non-uniqueness of weak Leray solution of Navier-Stokes equation on the three-dimensional  hyperbolic manifolds. Furthermore, Pierfelice \cite{Pi} has proved the dispersive and smoothing estimates for Stokes semigroups on the generalized non-compact manifolds with negative Ricci curvature then combines these estimates with Kato-iteration method to prove the existence and uniqueness of strong mild solutions to Navier-Stokes equations. The existence and stability of periodic and asymptotically almost periodic mild solutions to the Navier-Stokes equations and generalized parabolic evolution equations on noncompact manifolds with negative curvature tensors have been established in some recent works \cite{HuyXuan2021,HuyXuan2022,HuyVan2023,XVQ2023,XuanVan2023}. In the related works, the Navier-Stokes equations associated with Hodge-Laplace operator has been studied in several  manifolds, e.g., on two sphere \cite{Cao1999,Il1991}, on compact Riemannian manifolds \cite{Fa2018,Fa2020,Ko2008,MiTa2001,Sa}, or on the connected sums of $\mathbb{R}^3$ in \cite{Zha}. 

In this paper, we consider the wellposedness and exponential stability of mild solutions for Boussinesq system \eqref{BouE} with initial data $(u(0),\theta(0))$ in $L^p({\bf M};\,\Gamma(T{\bf M}))\times L^p({\bf M};\,\mathbb{R})$ for the case $1<p\leq d$. We will also revisit the existence of periodic mild solutions by using the stability result. This method is known as Serrin principle on non-compact Riemannian manifolds (for detailed method see \cite{HuyHa2022, HuyVan2022} for the case of Navier-Stokes equations and see \cite{Se1959} for original method). Our work extend some recent ones of Navier-Stokes equations on non-compact Einstein manifolds with negative Ricci curvatures \cite{HuyHa2022,HuyVan2022,HuyXuan2021,Pi}.

In particular, we first represent system \eqref{BouE} under the matrix intergral equation (see equation \eqref{BouMatrixEq} below). Then, we use the estimates for the semigroups generating by Ebin-Marsden's Laplace and Laplace-Beltrami operators (obtained in \cite{Pi}) to prove the $L^p-L^q$-dispersive and smoothing estimates for the matrix semigroup asscociated with the Boussinesq system (see Lemma \ref{dispersive}). Using these estimates we prove the existence of bounded mild solution for the linear equation corresponding Boussinesq system (see Theorem \ref{thm1}). After that, we establish the estimates for the bilinear operator associated with Boussinesq system, i.e., bilinear estimates \eqref{Bilinear1}, \eqref{Bilinear2}. Combining these estimates with the existence for the linear equation and fixed point arguments we establish the existence of bounded mild solution for the Boussinesq system in Theorem \ref{PeriodicThm}. We use cone inequality to prove the exponential stability of the Boussinesq system (see Theorem \ref{stability}). Finally, we give an application of exponential stability to the existence of periodic solution (see Theorem \ref{pest}).

Note that, our results extend the ones obtained in \cite{Pi} in the aspect of equations and of the range of phase spaces's dimensions and in \cite{HuyHa2022} in the aspect of periodic solutions.

This paper is organized as follows: in Section \ref{S2}, we present the real hyperbolic space, some geometric operators and the setting of Boussinesq systems; in Section \ref{S3}, we give the $L^p-L^q$-dispersive and smoothing estimates and the proofs of the global existence of linear and semilinear equations; in Section \ref{S4}, we provide the exponential stability and the application to existence of periodic mild solution for the Boussinesq system; in Appendix \ref{app}, we discuss about the influence of gravitational fields to the well-posedness and give some boundedness of improper integrals.\\
{\bf Notations.}
Through this paper we use the following notations:\\
$\bullet$ The space of the bounded and continuous functions from $\r_+$ to the space $L^r(X)$ is denoted by $C_b(\r_+,L^r(X))$.\\
$\bullet$ The norm on the Cartesian product space $ L^r({\mathbf M};\Gamma(T{\mathbf M})) \times  L^r({\mathbf M};\mathbb{R}) $ is defined by 
$$\norm{(u,\theta)}_{L^r\times L^r} = \max\left\{ \norm{u}_{L^r({\mathbf M};\Gamma(T{\mathbf M}))},\norm{\theta}_{L^r({\mathbf M};\mathbb{R})}  \right\}.$$

\section{Boussinesq system on the real hyperbolic manifold}\label{S2}
Let  $({\bf M}=:\mathbb{H}^d(\mathbb{R}),g)$  be a real hyperbolic manifold of dimension $d\geq 2$ which is realized as the upper sheet 
$$x_0^2-x_1^2-x_2^2...-x_d^2 = 1 \,  \,( x_0\geq 1),$$
of hyperboloid in $\mathbb{R}^{d+1}$, equipped with the Riemannian metric 
$$g := -dx_0^2 + dx_1^2 + ... + dx_d^2.$$
In geodesic polar coordinates, the hyperbolic manifold is 
$$\mathbb{H}^d(\mathbb{R}): = \left\{ (\cosh \tau, \omega \sinh \tau), \, \tau\geq 0, \omega \in \mathbb{S}^{d-1}  \right\}$$
with the metric 
$$g := d\tau^2+(\sinh\tau)^2d\omega^2$$
where  $d\omega^2$ is the canonical metric on the sphere $\mathbb{S}^{d-1}$.  
A remarkable property on ${\bf M}$ is the Ricci curvature tensor : $\mathrm{Ric}_{ij}=-(d-1)g_{ij}$. We refer readers to the reference \cite{Pa} for more details about the hyperbolic geometry. 

In order to define Laplace operator on manifolds, Ebin and Marsden introduced the vectorial laplace $L$ on vector field $u$ by using the deformation tensor (see \cite{EbiMa} and more details in \cite{Tay,Pi}):
$$Lu := \frac{1}{2}\mathrm{div}(\nabla u + \nabla u^t)^{\sharp},$$
where $\omega^{\sharp}$ is a vector field associated with the 1-form $\omega$ by $g(\omega^{\sharp},Y) = \omega(Y) \, \forall Y \in \Gamma(T{\bf M})$.
Since $\mathrm{div}\, u=0$ , $L$ can be expressed as 
$$Lu = \overrightarrow{\Delta}u + R(u),$$
where $\overrightarrow{\Delta}u =- \nabla^*\nabla u= \mathrm{Tr}_g(\nabla^2u)$ is the Bochner-Laplace
and $R(u)=(\mathrm{Ric}(u,\cdot))^{\sharp}$ is the Ricci operator. Since $\mathrm{Ric}(u,\cdot)=-(d-1)g(u,\cdot)$, we have $R(u)=-(d-1)u$ and
$$Lu = \overrightarrow{\Delta}u -(d-1)u.$$
By using the Weitzenb\"ock formula on $1$-form $u^\flat$ (which is associated with $u$ by $g(u,Y) = u^\flat(Y), \, Y\in \Gamma(T{\bf M})$):
$$\Delta_H u^\flat = \nabla^*\nabla u^\flat + \mathrm{Ric}(u,.),$$
where $\Delta_H = d^*d+dd^*$ is the Hodge-Laplace on $1$-forms, we can also relate $L$ to
the Hodge-Laplace
$$Lu = \left( -\Delta_H u^\flat + 2\mathrm{Ric}(u,.)  \right)^\sharp.$$

For simplicity we consider the incompressible Boussinesq system on the real hyperbolic manifold ${\bf M}$ with the volume expansion coefficient $\kappa=1$:
\begin{equation}\label{BouEq} 
\left\{
  \begin{array}{rll}
 u_t + (u\cdot\nabla)u - L u + \nabla p \!\! &= \theta h + \dive F, \hfill \\
\nabla \cdot u \!\!&=\; 0, \\
\theta_t - \widetilde{L} \theta + (u \cdot \nabla)\theta \!\!&=\; \dive f, \\
u(0) \!\!& = \;u_0,\\
\theta(0) \!\!& = \;\theta_0,\\
\end{array}\right.
\end{equation}
where $L= -(d-1)+ \overrightarrow{\Delta}$ is Ebin-Marsden's Laplace operator, $\widetilde{L}= \Delta_g$ is Laplace-Beltrami operator associated with metric $g$. 
The functions $f: {\bf M}\times \mathbb{R} \to \Gamma(T{\bf M})$ is given such that $\dive f$ represents the reference temperature and $F: {\bf M} \times \mathbb{R} \to \Gamma(T{\bf M}\otimes T{\bf M})$ is a second order tensor fields such that $\dive F$ represents the external force.
The unknowns are $u(x,t): {\bf M} \times \r \to \Gamma(T{\bf M}), p(x,t): {\bf M} \times \r \to \mathbb{R}$ and $\theta(x,t): {\bf M} \times \r \to \mathbb{R}$ representing respectively, the velocity field, the pressure and the temperature of the fluid at point $(x,t) \in  {\bf M}\times \r$.
Normaly, the gravitational field $h$ does not depend on time (see \cite{Barrow2020} for the formula of gravitational fied on hyperbolic spaces). However, in this paper, we will consider a more general case, where $h:{\bf M}\times \mathbb{R}_+ \to \Gamma(T{\bf M})$ depends on time and satisfies the following assumption (see Appendix \ref{B} for the discussion of gravitational field) which guarantees the regularity for elliptic problem to determine the pressure $p$:
\begin{assumption}\label{Assum}
Assume that function $h(\cdot,t)$ satisfies
\begin{equation}
h\in C_b(\mathbb{R}_+, L^\infty(\Gamma(T{\bf M}))) \hbox{   and   } h\in C_b(\mathbb{R}_+, L^{\frac{d}{2},\infty}(\Gamma(T{\bf M}))).
\end{equation}
\end{assumption}

Taking divergence to the first equation of system \eqref{BouEq1}, we get
\begin{equation}\label{EllipticEq}
\Delta_g p = \dive[-\dive(u\otimes u) + \theta h + \dive F]. 
\end{equation} 
If we consider $u(\cdot,t) \in L^{p}({\bf M};\Gamma(T{\bf M}))$, $\theta(\cdot,t) \in L^p({\bf M}; \mathbb{R})$, $h(\cdot,t)\in L^{\infty}({\bf M};\Gamma(T{\bf M}))$ and $F(\cdot,t)\in L^{p/2}({\bf M};\Gamma(T{\bf M}\otimes T{\bf M}))$, then we have $-\dive(u\otimes u) + \theta h + \dive F \in L^p({\bf M};\Gamma(T{\bf M}))$. 
Moreover, the spectral of $\Delta_g$ on hyperbolic manifold ${\bf M} = \mathbb{H}^d(\mathbb{R})$ is $\left(-\infty, -\frac{(d-1)^2}{4}\right]$ which does not contain $0$, then operator $\Delta_g \,: \, W^{2,q}({\bf M};\mathbb{R}) \rightarrow L^q({\bf M};\mathbb{R})$ is an isomorphism for $2\leq q <\infty$. Therefore, for $p>1$, we can choose the solution of elliptic equation \eqref{EllipticEq} by
\begin{equation}
p= \Delta_g^{-1}\dive[-\dive(u\otimes u) + \theta h + \dive F].
\end{equation}
Hence
\begin{equation}
\nabla p= \nabla(-\Delta_g)^{-1}\dive[\dive(u\otimes u) - \theta h - \dive F].
\end{equation}
Since Riesz transforms are $L^p$-bounded on real hyperbolic manifolds (see \cite{Loho}), we obtain that the operator $\nabla(-\Delta_g)^{-1} \mathrm{div}:L^p({\bf M};\Gamma(T{\bf M})) \to L^p({\bf M};\Gamma(T{\bf M}))$ is bounded. Therefore, we have $\nabla p\in L^p({\bf M};\Gamma(T{\bf M}))$.

Since we have
$\dive(\theta u) = \theta \dive u + (\nabla \theta)\cdot u = (u\cdot \nabla)\theta$ for $\dive u=0$, the Boussinesq system \eqref{BouEq} can be rewritten as
\begin{equation}\label{BouEq1} 
\left\{
  \begin{array}{rll}
 u_t + \dive(u\otimes u) - L u + \nabla p \!\! &= \theta h + \dive F, \hfill \\
\nabla \cdot u \!\!&=\; 0, \\
\theta_t - \widetilde{L} \theta + \dive(\theta u) \!\!&=\; \dive f, \\
u(0,x) \!\!& = \;u_0(x) \in \Gamma(T{\bf M}),\\
\theta(0,x) \!\!& = \;\theta_0(x) \in \mathbb{R},\\
\end{array}\right.
\end{equation}
Applying the Kodaira-Hodge operator $\mathbb{P}:= I + \nabla(-\Delta_g)^{-1} \mathrm{div}$ to the system \eqref{BouEq1}, by the same manner to Navier-Stokes equation (see \cite{Pi}) we get
\begin{align}\label{AbstractE}
\begin{cases}
u_t &= Lu + \mathbb{P}(\theta h) + \mathbb{P}\dive (-u\otimes u + F),\\
\nabla\cdot u &=0\\
\theta_t &= \widetilde{L}\theta + \dive(-\theta u) + \dive f,\\
u(0) &= u_0,\, \theta(0) = \theta_0.
\end{cases}
\end{align}

Considering system \eqref{AbstractE} with $(u,\theta)$ in the Cartesian product space $C_b(\mathbb{R}_+,L^p({\bf M};\Gamma(T{\bf M})))\times C_b(\mathbb{R}_+,L^p({\bf M};\mathbb{R}))$. 
We set $\mathcal{A}:=%
\begin{bmatrix}
-L & 0\\
0 & -\widetilde{L}
\end{bmatrix}
$ acting on the space $L^p({\bf M};\Gamma(T{\bf M}))\times L^p({\bf M};\mathbb{R})$. By using Duhamel's principle in a matrix form, we get the following
integral formulation for (\ref{AbstractE}):
\begin{equation}\label{BouMatrixEq}
Z(t)
=e^{-t\A}Z_0
+B(Z,Z)  (t)+T_h(\theta)(t) + \mathbb{T} \left(
\begin{bmatrix}
F\\
f
\end{bmatrix}
\right)(t), 
\end{equation}
where $Z_0=(u_0,\theta_0)$, $Z(t)=(u(t),\theta(t))$ and the bilinear, linear-coupling and external forced operators used in the above equation
are given respectively by
\begin{equation}
B(Z_1,Z_2)  (t)=-\int_{0}^{t} e^{-(t-s)\A}\dive
\begin{bmatrix}
\mathbb{P}(u\otimes v)\\
u\xi
\end{bmatrix}
(s)ds, \, Z_1 = \begin{bmatrix}u\\\theta\end{bmatrix},\, Z_2 = \begin{bmatrix}v\\\xi\end{bmatrix} \label{Bilinear},
\end{equation}
 
\begin{equation}\label{LinearT}
T_h(\theta)(t)=\int_{0}^{t}e^{-(t-s)\A}%
\begin{bmatrix}
\mathbb{P}(\theta h)\\
0
\end{bmatrix}
(s)ds,\, 
\mathbb{T} \left(
\begin{bmatrix}
F\\
f
\end{bmatrix}
\right)(t)= \int_0^t e^{-(t-s)\A}\dive
\begin{bmatrix}
\mathbb{P} (F)\\
f
\end{bmatrix}
(s) ds.   
\end{equation}

\section{The global existence}\label{S3}

\subsection{Some useful estimates} 
To establish the well-posedness for equation \eqref{LinearE}, we first prove the $L^p-L^q$-dispersive and smoothing estimates for the matrix semigroup $e^{-t\mathcal{A}}$: 
\begin{lemma}\label{estimates}
\begin{itemize}
\item[$(i)$] For $t>0$, and $p$, $q$ such that $1\leq p \leq q \leq \infty$, the following dispersive estimates hold: 
\begin{equation}\label{dispersive}
\left\| e^{-t \A} Z_0\right\|_{L^q\times L^q} \leq [h_d(t)]^{\frac{1}{p}-\frac{1}{q}}e^{-t(\gamma_{p,q})}\left\| Z_0 \right\|_{L^p\times L^p} 
\end{equation}
for all $Z_0=(u_0,\theta_0) \in L^p({\mathbf M};\Gamma(T{\mathbf M})) \times L^p({\mathbf M};\mathbb{R})$, 
 where 
 $$h_d(t): = C\max\left( \frac{1}{t^{d/2}},1 \right),\, 
   \gamma_{p,q}:=\frac{\delta_d}{2}\left[ \left(\frac{1}{p} - \frac{1}{q} \right) + \frac{8}{q}\left( 1 - \frac{1}{p} \right) \right]$$ 
and $\delta_d$ are positive constants depending only on $d$.  
\item[$(ii)$] For $p$ and $q$ such that $1<p\leq q <\infty$ we obtain for all $t>0$ that
\begin{equation}
\left\| e^{-t\A}\mathrm{div} Z_0^\sharp \right\|_{L^q\times L^q} \leq [h_d(t)]^{\frac{1}{p}-\frac{1}{q}+\frac{1}{d}}e^{-t\left( \frac{\gamma_{q,q}+\gamma_{p,q}}{2} \right)} \left\| Z_0^\sharp \right\|_{L^p\times L^p},
\end{equation}
where $Z_0^\sharp = (T^{\sharp}_0,U^{\sharp}_0)$, for all tensor $T^{\sharp}_0 \in L^p({\mathbf M};\Gamma(T{\mathbf M}\otimes T{\mathbf M}))$ and all vector field $U_0^{\sharp} \in L^p({\mathbf M}; \Gamma(T{\mathbf M}))$. 
\end{itemize}
\end{lemma}
\begin{proof}
We use the fact that
$ e^{-t\A} =  \begin{bmatrix}
e^{tL}&&0\\
0&& e^{t\widetilde{L}}
\end{bmatrix}$   
and the $L^p-L^q$-dispersive and smoothing estimates of the semigroup $e^{tL}$ (associated with Ebin-Marsden's Laplace operator) and the heat semigroup $e^{t\widetilde{L}}$ (associated with Laplace-Beltrami operator $\Delta_g$) which are proved by Pierfelice \cite{Pi}. In particular, assertion $i)$ is valid since
 the fact that: for $t>0$, and $p$, $q$ such that $1\leq p \leq q \leq \infty$, 
the following $L^p-L^q$-dispersive estimates hold (see \cite[Theorem 4.1]{Pi} and its proof): 
\begin{equation}
\left\| e^{t L} u_0\right\|_{L^q} \leq [h_d(t)]^{\frac{1}{p}-\frac{1}{q}}e^{-t(d-1 + \gamma_{p,q})}\left\| u_0 \right\|_{L^p} \leq [h_d(t)]^{\frac{1}{p}-\frac{1}{q}}e^{-t(\gamma_{p,q})}\left\| u_0 \right\|_{L^p}
\end{equation}
for all $u_0 \in L^p({\mathbf M};\Gamma(T{\mathbf M}))$ and
\begin{equation}
\left\| e^{t\widetilde{L} } \theta_0\right\|_{L^q} \leq [h_d(t)]^{\frac{1}{p}-\frac{1}{q}}e^{-t( \gamma_{p,q})}\left\| \theta_0 \right\|_{L^p}\hbox{ for all }\theta_0 \in L^p({\mathbf M};\mathbb{R}),    
\end{equation}
where $h_d(t): = C\max\left( \frac{1}{t^{d/2}},1 \right)$, 
   $\gamma_{p,q}:=\frac{\delta_d}{2}\left[ \left(\frac{1}{p} - \frac{1}{q} \right) + \frac{8}{q}\left( 1 - \frac{1}{p} \right) \right]$ and $\delta_d$ are positive constants depending only on $d$.  

Assertion $ii)$ comes from the following $L^p-L^q$-smoothing estimates: for $1<p\leq q <\infty$ and $t>0$ we have (see \cite[Corollary 4.3]{Pi} and its proof):
\begin{eqnarray}
\left\| e^{tL}\mathrm{div}T^{\sharp}_0 \right\|_{L^q} &\leq& [h_d(t)]^{\frac{1}{p}-\frac{1}{q}+\frac{1}{d}}e^{-t\left(d-1 + \frac{\gamma_{q,q}+\gamma_{p,q}}{2} \right)} \left\| T_0^{\sharp}\right\|_{L^p}\cr
&\leq& [h_d(t)]^{\frac{1}{p}-\frac{1}{q}+\frac{1}{d}}e^{-t\left( \frac{\gamma_{q,q}+\gamma_{p,q}}{2} \right)} \left\| T_0^{\sharp}\right\|_{L^p}
\end{eqnarray}
and
\begin{equation}
\left\| e^{t\widetilde{L}}\mathrm{div}U^{\sharp}_0 \right\|_{L^q} \leq [h_d(t)]^{\frac{1}{p}-\frac{1}{q}+\frac{1}{d}}e^{-t\left(\frac{\gamma_{q,q}+\gamma_{p,q}}{2} \right)} \left\| U_0^{\sharp}\right\|_{L^p}    
\end{equation}
for all tensor $T^{\sharp}_0 \in L^p({\mathbf M};\Gamma(T{\mathbf M}\otimes T{\mathbf M}))$ and vector field $U_0^\sharp \in L^p({\mathbf M}; T{\mathbf M})$. 
\end{proof}

Setting  $\mathbb{L}^\eta({\bf M}) = L^\eta(\mathbf M;\Gamma(T\mathbf{M}))\times L^\eta(\mathbf M;\r))$, where $\eta>0$. 
Using Lemma \ref{estimates}, we obtain the some useful linear estimates for $T_h(\cdot)$ in the following lemma. 
\begin{lemma}\label{LinearEst1}
Let $({\bf M},g)$ be a $d$-dimensional real hyperbolic manifold with $d\geqslant 2$ and $1<p\leq d$, $0<\delta<1$. The following assertion hold

(i) For $h \in C_b(\r_+,L^{\frac{d}{\delta}}({\bf M}; \Gamma(T{\bf M})))$ and $\eta\in C_b(\r_+,L^p({\bf M},\mathbb{R}))$, we have for all $t>0$:
\begin{equation}\label{Li1}
\norm{T_h(\eta)(t)}_{\mathbb{L}^p} \leqslant N_1\norm{h}_{\infty,L^{\frac{d}{\delta}}}\norm{\eta}_{\infty,L^p}.
\end{equation}
	
(ii) For $h \in C_b(\r_+,L^{\frac{d}{\delta}}({\bf M}; \Gamma(T{\bf M})))$ and $\eta\in C_b(\r_+,L^d({\bf M},\mathbb{R}))$, we have for all $t>0$:
\begin{equation}\label{Li2}
\norm{T_h(\eta)(t)}_{\mathbb{L}^d} \leqslant N_2\norm{h}_{\infty,L^{\frac{d}{\delta}}}\norm{\eta}_{\infty,L^d}.
\end{equation}
	
(iii) For $h \in C_b(\r_+,L^{\frac{d}{\delta}}({\bf M}; \Gamma(T{\bf M})))$ and $\eta\in C_b(\r_+,L^{\frac{d}{\delta}}({\bf M},\mathbb{R}))$, we have for all $t>0$:
\begin{equation}\label{Li3}
[h_d(t)]^{-\frac{1-\delta}{d}}e^{\alpha t}\norm{T_h(\eta)(t)}_{\mathbb{L}^{\frac{d}{\delta}}} \leqslant N_3\norm{h}_{\infty,L^{\frac{d}{\delta}}}\sup_{t>0}\left([h_d(t)]^{-\frac{1-\delta}{d}}e^{\alpha t}\norm{\eta(t)}_{L^{\frac{d}{\delta}}}\right).
\end{equation}
Here, the positive constants $N_1,N_2$ and $N_3$ are not dependent on $h$ and $\eta$.
\end{lemma}
\begin{proof}
\item[$(i)$] Using formula of $T_h(\cdot)$ in \eqref{LinearT}, Lemma \ref{estimates}$(i)$ and Holder's inequality, we can estimate
\begin{eqnarray}
\norm{T_h(\eta)(t)}_{\mathbb{L}^p}&\leq&  \int_0^t \norm{e^{-(t-\tau)\cal{A}}\begin{bmatrix}\mathbb{P}(h\eta)(\tau)\\0\end{bmatrix}}_{\mathbb{L}^p} d\tau\cr
&\leq&\int_0^t [h_d(t-\tau)]^{\frac{\delta}{d}}e^{-\beta_1(t-\tau)}\norm{\begin{bmatrix}(h\eta)(\tau)\\0\end{bmatrix}}_{\mathbb{L}^{\frac{dp}{d+\delta p}}}d\tau\cr
&\leq& \int_0^t [h_d(t-\tau)]^{\frac{\delta}{d}}e^{-\beta_1(t-\tau)}\norm{h(\tau)}_{L^{\frac{d}{\delta}}}\norm{\eta(\tau)}_{L^{p}} d\tau\cr	
&\leq& \norm{h}_{\infty,L^{\frac{d}{\delta}}}\norm{\eta}_{\infty,L^p} \int_0^t [h_d(t-\tau)]^{\frac{\delta}{d}}e^{-\beta_1(t-\tau)}d\tau\cr
&\leq& \norm{h}_{\infty,L^{\frac{d}{\delta}}}\norm{\eta}_{\infty,L^p} \int_0^t C^{\frac{\delta}{d}}\left[  (t-\tau)^{-\frac{\delta}{2}}+1\right]  e^{-\beta_1(t-\tau)}d\tau\cr
&\leq&  \norm{h}_{\infty,L^{\frac{d}{\delta}}}\norm{\eta}_{\infty,L^p}  C^{\frac{\delta}{d}}\left[\beta_1^{\frac{\delta}{2}-1}{\bf\Gamma}\left( 1-\frac{\delta}{2}\right)  +\dfrac{1}{\beta_1} \right] \cr
&\leqslant&  N_1 \norm{h}_{\infty,L^{\frac{d}{\delta}}}\norm{\eta}_{\infty,L^p},
	\end{eqnarray}
where $\beta_1= \gamma_{dp/(1+\delta p),p}$, and $N_1= C^{\frac{\delta}{d}}\left[\beta_1^{\frac{\delta}{2}-1}{\bf\Gamma}\left( 1-\frac{\delta}{2}\right)  +\dfrac{1}{\beta_1} \right]$.

\item[$(ii)$] The second assertion is prove by the same manner as in Assertion $(i)$ with the constants $\beta_1$ and $N_1$ replaced by $\hat{\beta}_1= \gamma_{d/(1+\delta),d}$ and 
$N_2 = C^{\frac{\delta}{d}}\left[\hat{\beta}_1^{\frac{\delta}{2}-1}{\bf\Gamma}\left( 1-\frac{\delta}{2}\right)  +\dfrac{1}{\hat{\beta}_1} \right]$.

\item[$(iii)$] Now, we prove the third assertion. Using again Lemma \ref{estimates}$(i)$ and Holder's inequality, we can obtain  that
\begin{eqnarray}\label{ine2}
[h_d(t)]^{-\frac{1-\delta}{d}}e^{\alpha t} \norm{T_h(\eta)(t)}_{\mathbb{L}^{\frac{d}{\delta}}}
&\leq& [h_d(t)]^{-\frac{1-\delta}{d}} e^{\alpha t}\int_0^t \norm{e^{-(t-\tau)\cal{A}}\begin{bmatrix}\mathbb{P}(h\eta)(\tau)\\0\end{bmatrix}}_{\mathbb{L}^{\frac{d}{\delta}}}d\tau\cr
&\leq& [h_d(t)]^{-\frac{1-\delta}{d}} e^{\alpha t}\int_0^t [h_d(t-\tau)]^{\frac{\delta}{d}}e^{-\tilde{\beta}_1(t-\tau)}\norm{\begin{bmatrix}(h\eta)(\tau)\\0\end{bmatrix}}_{\mathbb{L}^{\frac{d}{2\delta}}}d\tau\cr
&\leq& [h_d(t)]^{-\frac{1-\delta}{d}} e^{\alpha t} \int_0^t [h_d(t-\tau)]^{\frac{\delta}{d}}e^{-\tilde{\beta}_1(t-\tau)}\norm{h(\tau)}_{L^{\frac{d}{\delta}}}\norm{\eta(\tau)}_{L^{\frac{d}{\delta}}}d\tau\cr
&\leq& \norm{h}_{\infty,L^{\frac{d}{\delta}}}\sup_{t>0}\left([h_d(t)]^{-\frac{1-\delta}{d}}e^{\alpha t}\norm{\eta(t)}_{L^{\frac{d}{\delta}}}\right)\cr
&&\times \int_0^t [h_d(t-\tau)]^{\frac{\delta}{d}}[h_d(\tau)]^{\frac{1-\delta}{d}}e^{-({\tilde{\beta}_1}-\alpha)(t-\tau)} d\tau\cr
&\leq& N_3 \norm{h}_{\infty,L^{\frac{d}{\delta}}} \sup_{t>0}\left([h_d(t)]^{-\frac{1-\delta}{d}}e^{\alpha t}\norm{\eta(t)}_{L^{\frac{d}{\delta}}}\right),
\end{eqnarray}
where $\tilde{\beta}_1= \gamma_{d/2\delta,d/\delta}$ and $N_3 = \int_0^t [h_d(t-\tau)]^{\frac{\delta}{d}}[h_d(\tau)]^{\frac{1-\delta}{d}}e^{-({\tilde{\beta}_1}-\alpha)(t-\tau)} d\tau<+\infty$ (see Appendix \ref{app} for the boundedness of this improper integral).
\end{proof}

Setting $\mathcal{L}^\eta({\bf M}) = L^\eta(\mathbf M;\Gamma(T\mathbf{M}\otimes T\mathbf{M}) \times L^\eta (\mathbf M;\r\times\r)$.
The similar estimates as in Lemma \ref{LinearEst1} are valid for $\mathbb{T}(\cdot)$:
\begin{lemma}\label{LinearEst2}
Let $({\bf M},g)$ be a $d$-dimensional real hyperbolic manifold with $d\geqslant 2$ and $1<p\leq d$, $0<\delta<1$. The following assertion hold

(i) For $(F,f)\in C_b(\r_+, \mathcal{L}^{\frac{dp}{d+\delta p}}({\bf M}))$, we have for all $t>0$:
\begin{equation}\label{Li11}
\norm{\mathbb{T}(F,f)(t)}_{\mathbb{L}^p} \leqslant M_1\norm{(F,f)}_{\infty, \mathcal{L}^{\frac{dp}{d+\delta p}}}.
\end{equation}
	
(ii) For $(F,f)\in C_b(\r_+, \mathcal{L}^{\frac{d}{1+\delta}}({\bf M}))$, we have for all $t>0$:
\begin{equation}\label{Li22}
\norm{\mathbb{T}(F,f)(t)}_{\mathbb{L}^d} \leqslant M_2\norm{(F,f)}_{\infty, \mathcal{L}^{\frac{d}{1+\delta}}}.
\end{equation}
	
(iii) For $(F,f)\in C_b(\r_+, \mathcal{L}^{\frac{d}{2\delta}}({\bf M}))$, we have for all $t>0$:
\begin{equation}\label{Li33}
[h_d(t)]^{-\frac{1-\delta}{d}}e^{\alpha t}\norm{\mathbb{T}(F,f)(t)}_{\mathbb{L}^{\frac{d}{\delta}}} \leqslant M_3\sup_{t>0}\left( [h_d(t)]^{-\frac{1-\delta}{d}}e^{\alpha t} \norm{(F,f)(t)}_{\mathcal{L}^{\frac{d}{2\delta}}} \right).
\end{equation}
Here, the positive constants $R_1,R_2$ and $R_3$ are not dependent on $F$ and $f$.
\end{lemma}
\begin{proof}
\item[$(i)$] Using formula of $\mathbb{T}(\cdot)$ in \eqref{LinearT}, the boundedness of $\mathbb{P}$ (see \cite{Loho}) and Lemma \ref{estimates}$(ii)$, we have the following estimates
\begin{eqnarray}
\norm{\mathbb{T}\left( \begin{bmatrix}
	F\\f
	\end{bmatrix}\right) (t)}_{\mathbb{L}^p} 
&\leq& \int_0^t \norm{e^{-(t-\tau)\cal{A}}\dive \left( \begin{bmatrix}
	\mathbb{P}(F)\\f
	\end{bmatrix}\right)(\tau)}_{\mathbb{L}^p} d\tau\cr
&\leq& \int_0^t [h_d(t-\tau)]^{\frac{\delta}{d}+\frac{1}{d}}e^{-\beta_2(t-\tau)} \left\|\begin{bmatrix}
	F\\f
	\end{bmatrix} (\tau) \right\|_{\mathcal{L}^{\frac{dp}{d+\delta p}}} d\tau\cr
&\leq&  \int_0^t [h_d(t-\tau)]^{\frac{\delta}{d}+\frac{1}{d}}e^{-\beta_2(t-\tau)}  d\tau \norm{ \begin{bmatrix}
			F\\f
			\end{bmatrix}}_{\infty, \mathcal{L}^{\frac{dp}{d+\delta p}}}\cr
&\leq& \int_0^t C^{\frac{\delta+1}{d}}\left[  (t-\tau)^{-\frac{\delta+1}{2}}+1\right]  e^{-\beta_2(t-\tau)}  d\tau \norm{ \begin{bmatrix}
				F\\f
		\end{bmatrix}}_{\infty, \mathcal{L}^{\frac{dp}{d+\delta p}}}\cr
&\leq&   C^{\frac{\delta+1}{d}}\left[\beta_2^{\frac{\delta-1}{2}}{\bf\Gamma}\left( \frac{1}{2}-\frac{\delta}{2}\right)  +\dfrac{1}{\beta_2} \right]  \norm{ \begin{bmatrix}
			F\\f
	\end{bmatrix}}_{\infty, \mathcal{L}^{\frac{dp}{d+\delta p}}}	\cr
&\leq&  M_1\norm{\begin{bmatrix}
	F\\ f
\end{bmatrix}}_{\infty, \mathcal{L}^{\frac{dp}{d+\delta p}}},
	\end{eqnarray}
where $\beta_2= \frac{\gamma_{p,p} + \gamma_{dp/(1+\delta p),p}}{2}$ and $M_1=C^{\frac{\delta+1}{d}}\left[\beta_2^{\frac{\delta-1}{2}}{\bf\Gamma}\left( \frac{1}{2}-\frac{\delta}{2}\right)  +\dfrac{1}{\beta_2} \right].$

\item[$(ii)$] The proof is done by the same way as the one for Assertion $(i)$ with $\beta_2$ and $M_1$ replaced by $\hat{\beta}_2= \frac{\gamma_{d,d}+\gamma_{d/(1+\delta),d}}{2}$ and
$M_2=C^{\frac{\delta+1}{d}}\left[\hat{\beta}_2^{\frac{\delta-1}{2}}{\bf\Gamma}\left( \frac{1}{2}-\frac{\delta}{2}\right)  +\dfrac{1}{\hat{\beta}_2} \right].$

\item[$(iii)$] Using again Lemma \ref{estimates}$(ii)$ and the boundedness of $\mathbb{P}$, we can estimate
\begin{eqnarray}
&&[h_d(t)]^{-\frac{1-\delta}{d}}e^{\alpha t}\norm{\mathbb{T}\left( \begin{bmatrix}
			F\\f
		\end{bmatrix}\right) (t)}_{\mathbb{L}^{\frac{d}{\delta}}} \cr
&\leq& [h_d(t)]^{-\frac{1-\delta}{d}}e^{\alpha t}\int_0^t \norm{e^{-(t-\tau)\cal{A}}\dive \left( \begin{bmatrix}
			\mathbb{P}(F)\\f
		\end{bmatrix}\right)(\tau)}_{\mathbb{L}^{\frac{d}{\delta}}} d\tau\cr
&\leqslant& [h_d(t)]^{-\frac{1-\delta}{d}}e^{\alpha t}\int_0^t [h_d(t-\tau)]^{\frac{\delta}{d}+\frac{1}{d}}e^{-\tilde{\beta}_2(t-\tau)} \left\|\begin{bmatrix}
		F\\f
	\end{bmatrix} (\tau) \right\|_{\mathcal{L}^{\frac{d}{2\delta}}} d\tau\cr
&\leqslant&  \sup_{t>0}\left([h_d(t)]^{-\frac{1-\delta}{d}}e^{\alpha t}\norm{(F,f)(t)}_{\mathcal{L}^{\frac{d}{2\delta}}}\right) \int_0^t [h_d(t-\tau)]^{\frac{\delta+1}{d}}[h_d(\tau)]^{\frac{1-\delta}{d}}e^{-(\tilde{\beta}_2-\alpha)(t-\tau)} d\tau \cr 
&\leqslant& M_3\norm{\begin{bmatrix}
		F\\ f
\end{bmatrix}}_{\infty,\mathcal{L}^{\frac{d}{2\delta}}}.
\end{eqnarray}
where  $\tilde{\beta}_2= \frac{\gamma_{d/\delta,d/\delta} + \gamma_{d/2\delta,d/\delta}}{2}$ and $M_3 = \int_0^t [h_d(t-\tau)]^{\frac{\delta+1}{d}}[h_d(\tau)]^{\frac{1-\delta}{d}}e^{-(\tilde{\beta}_2-\alpha)(t-\tau)} d\tau<+\infty$ (see Appendix \ref{app} for the convergence of this improper integral).

\end{proof} 
 
\subsection{Bounded mild solutions for the linear equations}
For a given function $\eta$, we consider the following linear equation corresponding to the integral matrix equation \eqref{BouMatrixEq}: 
\begin{equation}\label{LinearE}
Z(t)=e^{-t\A}Z_0
+T_h(\eta)(t) + \mathbb{T} \left( \begin{bmatrix}
F\\
f
\end{bmatrix} \right) (t), 
\end{equation}
where $T_h(\cdot)$ and $\mathbb{T}(\cdot)$ are given by \eqref{LinearT}.

For $1<p\leq d$ and $0<\delta<1$, we establish the global well-posedness of mild solution on the half time-line axis to linear equation \eqref{LinearE} (and also semilinear equation \eqref{BouMatrixEq}) on the following space
 \begin{eqnarray*}
 	\mathcal{X}&=&\left\{Z\in C_b(\r_+, \mathbb{L}^p(\mathbf M) \cap \mathbb{L}^d(\mathbf M) \cap \mathbb{L}^{d/\delta}(\mathbf M)): \sup\limits_{t >0}\norm{Z(t)}^\blacklozenge  < +\infty \right\},
 \end{eqnarray*}
where $\norm{Z(t)}^\blacklozenge  =\norm{Z(t)}_{\mathbb{L}^p} + \norm{Z(t)}_{\mathbb{L}^d}  +[h_d(t)]^{-\frac{1-\delta}{d}}e^{\alpha t} \norm{Z(t)}_{\mathbb{L}^{\frac{d}{\delta}}}$. Clearly, the space $\mathcal{X}$ is a Banach space
equipped with the norm 
\begin{equation}\label{space1}
\norm{Z}_{\mathcal{X}} = \sup_{t>0}\norm{Z(t)}^\blacklozenge = \norm{Z}_{\infty,\mathbb{L}^p} + \norm{Z}_{\infty,\mathbb{L}^d} + \sup_{t>0}\left( [h_d(t)]^{-\frac{1-\delta}{d}}e^{\alpha t} \norm{Z(t)}_{\mathbb{L}^{\frac{d}{\delta}}} \right).
\end{equation}
For this purpose, we consider the external force $(F,f)$ in the following Banach space
$$\mathcal{Y} = \left\{ (F,f)\in \mathcal{Y} = C_b(\r_+, \mathcal{L}^{\frac{dp}{d+\delta p}}({\bf M})\cap \mathcal{L}^{\frac{d}{1+\delta}}({\bf M})\cap \mathcal{L}^{\frac{d}{2\delta}}({\bf M})): \sup_{t>0}\left( [h_d(t)]^{-\frac{1-\delta}{d}}e^{\alpha t}\norm{(F,f)}_{\mathcal{L}^{\frac{d}{2\delta}}} \right)<+\infty  \right\},$$
equipped with the norm
\begin{equation}
\norm{(F,f)}_{\mathcal{Y}} = \norm{(F,f)}_{\infty,\mathcal{L}^{\frac{dp}{d+\delta p}}} + \norm{(F,f)}_{\infty,\mathcal{L}^{\frac{d}{1+\delta}}} + \sup_{t>0}\left( [h_d(t)]^{-\frac{1-\delta}{d}}e^{\alpha t}\norm{(F,f)}_{\mathcal{L}^{\frac{d}{2\delta}}} \right). 
\end{equation}
\begin{remark}
Note that, the space $\mathcal{X}$ guarantees well-posedness in this work, which is based on lower-dimensional $L^p$-spaces in comparing  with the dimension of manifold $({\bf M},g)$, i.e., $p \leq d$. In a previous work \cite{Pi}, Pierfelice established the well-posedness of the Navier-Stokes equations (a specific case of the Boussinesq system when $\theta = 0$) in the space $X_T$, which is essentially constructed in a higher-dimensional setting, i.e., $p > d$ (for details, see the end of page 30 in \cite{Pi}). Therefore, our results, together with those obtained by Pierfelice, provide a complete range of dimensions for $L^p$-phase spaces that ensure the well-posedness of the Navier-Stokes equations on real hyperbolic spaces.
\end{remark}

The main theorem of this subsection is as follows:
\begin{theorem}\label{thm1}
Let $(\mathbf{M},g)$ be a $d$-dimensional real hyperbolic manifold with $d \geqslant 2$. Let $1<p\leq d$, $0<\delta<1$  and  $\;0<\alpha<\min\{\gamma_{d,d/\delta},\; \gamma_{d/2\delta,d/\delta},\; \frac{\gamma_{d/\delta,d/\delta}+\gamma_{d/(2\delta),d/\delta}}{2}\} $. Assume that the field $h$ satisfies Assumption \ref{Assum} and the external force $(F,f)\in \mathcal{Y}$. For a given initial data $Z_0=(u_0,\theta_0)\in \mathbb{L}^p({\bf M})\cap \mathbb{L}^d({\bf M})$ and  $(0,\eta) \in \mathcal{X}$, then linear equation \eqref{LinearE} has a unique bounded solution $Z\in \mathcal{X}$ satisfying 
\begin{equation}\label{Bounded}
 \norm{Z}_{\mathcal X} \leq \norm{Z_0}_{\mathbb{L}^p\cap \mathbb{L}^d} + N \left\| h \right\|_{\infty,\frac{d}{\delta}} \norm{(0,\eta)}_{\mathcal {X}} + M \norm{(F,f)}_{\mathcal{Y}},  
\end{equation}
 where the positive constants $M$ and $N$ are independent to  $h$, $\eta$, $F$ and $f$. 
\end{theorem}

\begin{proof}
We prove this theorem by using the dispersive estimates in Lemma \ref{estimates}$(i)$ and the linear estimates obtained in Lemma \ref{LinearEst1} and \ref{LinearEst1}. In particular,
from Assumption \ref{Assum} and interpolation inequality (see inequality (2.7) in \cite[Lemma 2.1]{Hi1997}), we obtain that $h \in C_b(\r_+,L^q(\Gamma(T{\bf M})))$ for $\frac{d}{2}<q\leqslant \infty$. Hence, we have $h \in C_b(\mathbb{R}_+, L^{\frac{d}{\delta}}({\bf M};\Gamma(T{\bf M})))$ for $0<\delta<1$.
Therefore, from above lemmas we  can estimate $Z(t)$ as follows
\begin{eqnarray}\label{ine1p}
\norm{Z(t)}_{\mathbb{L}^p} &\leq&  \norm{e^{-t\A} Z_0}_{\mathbb{L}^p} + \norm{T_h(\eta)(t)}_{\mathbb{L}^p}  +\norm{\mathbb{T}\left( \begin{bmatrix}
	F\\f
	\end{bmatrix}\right) (t)}_{\mathbb{L}^p} \cr
&\leq&  \norm{Z_0}_{\mathbb{L}^p}    + N_1 \norm{h}_{\infty,\frac{d}{\delta}}\norm{\eta}_{\infty, L^p}  +M_1\norm{\begin{bmatrix}
	F\\ f
\end{bmatrix}}_{\infty,\mathcal{L}^{\frac{dp}{d+\delta p}}}\cr
&\leq&  \norm{Z_0}_{\mathbb{L}^p}    + N_1 \norm{h}_{\infty,\frac{d}{\delta}}\norm{(0,\eta)}_{\mathcal{X}}  +M_1\norm{\begin{bmatrix}
	F\\ f
\end{bmatrix}}_{\mathcal{Y}}
	\end{eqnarray}
\begin{eqnarray}\label{ine1d}
\norm{Z(t)}_{\mathbb{L}^d} &\leq&  \norm{e^{-t\A} Z_0}_{\mathbb{L}^d} +  \norm{T_h(\eta)(t)}_{\mathbb{L}^d}  +\norm{\mathbb{T}\left( \begin{bmatrix}
			F\\f
		\end{bmatrix}\right) (t)}_{\mathbb{L}^d} \cr
&\leq&  \norm{Z_0}_{\mathbb{L}^d} +N_2\norm{h}_{\infty,L^{\frac{d}{\delta}}}\norm{\eta}_{\infty,L^d}  + M_2\norm{\begin{bmatrix}
			F\\ f
	\end{bmatrix}}_{\infty,\mathcal{L}^{\frac{d}{1+\delta}}}\cr
&\leq&  \norm{Z_0}_{\mathbb{L}^d} +N_2\norm{h}_{\infty,L^{\frac{d}{\delta}}}\norm{(0,\eta)}_{\mathcal{X}}  + M_2\norm{\begin{bmatrix}
			F\\ f
	\end{bmatrix}}_{\mathcal{Y}},	
\end{eqnarray}
and
\begin{eqnarray}\label{ine2}
[h_d(t)]^{-\frac{1-\delta}{d}}e^{\alpha t}\norm{Z(t)}_{\mathbb{L}^{\frac{d}{\delta}}}&\leq&  [h_d(t)]^{-\frac{1-\delta}{d}}e^{\alpha t}\left[ \norm{e^{-t\A} Z_0}_{\mathbb{L}^{\frac{d}{\delta}}}+  \norm{T_h(\eta)(t)}_{\mathbb{L}^{\frac{d}{\delta}}}  +\norm{\mathbb{T}\left( \begin{bmatrix}
			F\\f
		\end{bmatrix}\right) (t)}_{\mathbb{L}^{\frac{d}{\delta}}}\right]  \cr
&\leq& \norm{Z_0}_{\mathbb{L}^d} + N_3\norm{h}_{\infty,L^{\frac{d}{\delta}}}\sup_{t>0}\left([h_d(t)]^{-\frac{1-\delta}{d}}e^{\alpha t}\norm{\eta(t)}_{L^{\frac{d}{\delta}}} \right)\cr
&& + M_3\sup_{t>0}\left([h_d(t)]^{-\frac{1-\delta}{d}}e^{\alpha t}\norm{\begin{bmatrix}
		F\\ f
\end{bmatrix}}_{\infty, \mathcal{L}^{\frac{d}{2\delta}}} \right)\cr
&\leq& \norm{Z_0}_{\mathbb{L}^d} + N_3\norm{h}_{\infty,L^{\frac{d}{\delta}}}\norm{(0,\eta)}_{\mathcal{X}} + M_3\norm{\begin{bmatrix}
		F\\ f
\end{bmatrix}}_{\mathcal{Y}}.
\end{eqnarray}
By setting $N=N_1+N_2+N_3, M=M_1+M_2+M_3$ and combining the inequalities \eqref{ine1p}, \eqref{ine1d} and \eqref{ine2}, we obtain the boundedness \eqref{Bounded} which leads to the existence of solutions for linear equation \eqref{LinearE}. The uniqueness holds clearly. Our proof is completed.
\end{proof}

\subsection{Bounded mild solutions for Boussinesq systems}
To study bounded mild solutions for the semilinar equation \eqref{BouMatrixEq}, we need to estimate the bilinear operator $B(\cdot,\cdot)$ given by the formula \eqref{Bilinear}.
\begin{lemma}\label{BiE}
Let $({\bf M},g)$ be a $d$-dimensional real hyperbolic manifold with $d\geqslant 2$ and $1<p\leq d$. For $Z_1=(u,v)\in \mathcal{X},\, Z_2=(v,\xi)\in \mathcal{X}$, there exists a positive constant $K$ such that
\begin{equation}\label{Bilinear1}
\norm{B( Z_1,Z_2)(t)}_{\mathcal{X}} \leq K \norm{Z_1}_{\mathcal X}\norm{Z_2}_{\mathcal X}  
\end{equation}
for all $t>0$.	
\end{lemma}
\begin{proof}
The proof is similarly the one of boundedness of the linear operator $\mathbb{T}(\cdot)$ in Lemma \ref{LinearEst2} but we need more techniques by estimating the tensor product $u\otimes u$ and $u\xi$.
Using the boundedness of operator $\mathbb{P}$, the $L^p-L^q$-smoothing estimates in assertion $ii)$ of Lemma \ref{estimates} and H\"older's inequality we have that
\begin{eqnarray}\label{ine3}
\norm{B(Z_1,Z_2)(t)}_{\mathbb{L}^p} &=&\norm{B\left( \begin{bmatrix}
		u\\ \theta   
	\end{bmatrix}, \begin{bmatrix}
		v\\ \xi   
	\end{bmatrix}  \right)(t)}_{\mathbb{L}^p} 
\cr&&\leq \int_0^t  \norm{e^{-(t-\tau)\A}\dive \begin{bmatrix}
		\p(u\otimes v)(\tau)\\
		(u\xi)(\tau)
\end{bmatrix}}_{\mathbb{L}^p} d\tau  \cr
	&&\leq \int_0^t [h_d(t-\tau)]^{\frac{\delta+1}{d}}e^{-\beta_2(t-\tau)}\norm{\begin{bmatrix}(u\otimes v)(\tau)\\
			(u\xi)(\tau)\end{bmatrix}}_{\mathbb{L}^{\frac{dp}{d+\delta p}}}d\tau\cr
&&\leq\int_0^t [h_d(t-\tau)]^{\frac{\delta+1}{d}}e^{-\beta_2(t-\tau)}\norm{\begin{bmatrix}
		u(\tau)\\
		\theta(\tau)
\end{bmatrix}}_{\mathbb{L}^{\frac{d}{\delta}}}\norm{\begin{bmatrix}
		v(\tau)\\
		\xi(\tau)
\end{bmatrix}}_{\mathbb{L}^p} d\tau\cr
	&&\leq \norm{Z_1}_{\mathcal{X}}\norm{Z_2}_{\mathcal{X}}\int_0^t [h_d(t-\tau)]^{\frac{\delta+1}{d}}e^{-\beta_2(t-\tau)}[h_d(\tau)]^{\frac{1-\delta}{d}}e^{-\alpha \tau}d\tau\cr
	&&\leq K_1\norm{Z_1}_{\mathcal{X}}\norm{Z_2}_{\mathcal{X}},
\end{eqnarray}
where $\beta_2=\dfrac{\gamma_{p,p}+\gamma_{dp/(1+\delta p),p}}{2} $ and $K_1=\int_0^t [h_d(t-\tau)]^{\frac{\delta+1}{d}}e^{-\beta_2(t-\tau)}[h_d(\tau)]^{\frac{1-\delta}{d}}e^{-\alpha \tau}d\tau<+\infty$ (see Appendix \ref{app} for this boundedness).

By the same manner, we can estimate
\begin{eqnarray}\label{ine33}
\norm{B(Z_1,Z_2)(t)}_{\mathbb{L}^d} 	&=&\norm{B\left( \begin{bmatrix}
			u\\ \theta   
		\end{bmatrix}, \begin{bmatrix}
			v\\ \xi   
		\end{bmatrix}  \right)(t)}_{\mathbb{L}^d} 
\cr&&\leq \int_0^t  \norm{e^{-(t-\tau)\A}\dive \begin{bmatrix}
			\p(u\otimes v)(\tau)\\
			(u\xi)(\tau)
	\end{bmatrix}}_{\mathbb{L}^d} d\tau  \cr
	&&\leq \int_0^t [h_d(t-\tau)]^{\frac{1+\delta}{d}}e^{-\hat{\beta}_2(t-\tau)}\norm{\begin{bmatrix}(u\otimes v)(\tau)\\
			(u\xi)(\tau)\end{bmatrix}}_{\mathbb{L}^{\frac{d}{1+\delta}}}d\tau\cr
	&&\leq\int_0^t [h_d(t-\tau)]^{\frac{1+\delta}{d}}e^{-\hat{\beta}_2(t-\tau)}\norm{\begin{bmatrix}
			u(\tau)\\
			\theta(\tau)
	\end{bmatrix}}_{\mathbb{L}^{\frac{d}{\delta}}}\norm{\begin{bmatrix}
			v(\tau)\\
			\xi(\tau)
	\end{bmatrix}}_{\mathbb{L}^d} d\tau\cr
	&&\leq \norm{Z_1}_{\mathcal{X}}\norm{Z_2}_{\mathcal{X}}\int_0^t [h_d(t-\tau)]^{\frac{1+\delta}{d}}e^{-\hat{\beta}_2(t-\tau)}[h_d(\tau)]^{\frac{1-\delta}{d}}e^{-\alpha \tau}d\tau\cr
	&&\leq K_2\norm{Z_1}_{\mathcal{X}}\norm{Z_2}_{\mathcal{X}},
\end{eqnarray}
where $\hat{\beta}_2=\dfrac{\gamma_{d,d}+\gamma_{d/(1+\delta),d}}{2} $ and $K_2=\int_0^t [h_d(t-\tau)]^{\frac{1+\delta}{d}}e^{-\hat{\beta}_2(t-\tau)}[h_d(\tau)]^{\frac{1-\delta}{d}}e^{-\alpha \tau}d\tau<+\infty$ (see Appendix \ref{app} for this boundedness).

Lastly, we use the fact that $[h_d(t)]^{-\frac{1-\delta}{d}}\le C^{\frac{\delta-1}{d}}$ (for $t>0$) to estimate
\begin{eqnarray}\label{ine333}
&&[h_d(t)]^{-\frac{1-\delta}{d}}e^{\alpha t}\norm{B(Z_1,Z_2)(t)}_{\mathbb{L}^{\frac{d}{\delta}}}	\cr
&=&[h_d(t)]^{-\frac{1-\delta}{d}}e^{\alpha t}\norm{B\left( \begin{bmatrix}
			u\\ \theta   
		\end{bmatrix}, \begin{bmatrix}
			v\\ \xi   
		\end{bmatrix}  \right)(t)}_{\mathbb{L}^{\frac{d}{\delta}}}\cr
&\leq& C^{\frac{\delta-1}{d}}e^{\alpha t}\int_0^t  \norm{e^{-(t-\tau)\A}\dive \begin{bmatrix}
			\p(u\otimes v)(\tau)\\
			(u\xi)(\tau)
	\end{bmatrix}}_{\mathbb{L}^{\frac{d}{\delta}}}d\tau\cr
&\leq& C^{\frac{\delta-1}{d}}e^{\alpha t}\int_0^t [h_d(t-\tau)]^{\frac{\delta+1}{d}}e^{-\tilde{\beta}_2(t-\tau)}\norm{\begin{bmatrix}(u\otimes v)(\tau)\\
			(u\xi)(\tau)\end{bmatrix}}_{\mathbb{L}^{\frac{d}{2\delta}}}d\tau
	\cr
&\leq& C^{\frac{\delta-1}{d}}e^{\alpha t} \int_0^t [h_d(t-\tau)]^{\frac{\delta+1}{d}}e^{-\tilde{\beta}_2(t-\tau)}\norm{\begin{bmatrix}u(\tau)\\
		\theta(\tau)\end{bmatrix}}_{\mathbb{L}^{\frac{d}{\delta}}}\norm{\begin{bmatrix}v(\tau)\\  \xi(\tau)\end{bmatrix}}_{\mathbb{L}^{\frac{d}{\delta}}}d\tau
	 \cr
&\leq& C^{\frac{\delta-1}{d}}\norm{Z_1}_{\mathcal X}\norm{Z_2}_{\mathcal X} \int_0^t [h_d(t-\tau)]^{\frac{\delta+1}{d}}[h_d(\tau)]^{\frac{2(1-\delta)}{d}}e^{-(\tilde{\beta}_2-\alpha)(t-\tau)} e^{-\alpha\tau}d\tau\cr
&\leq& K_3 \norm{Z_1}_{\mathcal X}\norm{Z_2}_{\mathcal X},
\end{eqnarray}
where $\tilde{\beta}_2=\dfrac{\gamma_{d/\delta,d/\delta}+\gamma_{d/2\delta,d/\delta}}{2} $ and we used the convergence (see Appendix \ref{app} for this boundedness):
$$K_3=C^{\frac{\delta-1}{d}}\int_0^t [h_d(t-\tau)]^{\frac{\delta+1}{d}}[h_d(\tau)]^{\frac{2(1-\delta)}{d}}e^{-(\tilde{\beta}_2-\alpha)(t-\tau)} e^{-\alpha\tau}d\tau<+\infty.$$
Therefore, by setting $K=\max \{K_1, K_2, K_3\}$, we obtain \eqref{Bilinear1} from inequalities \eqref{ine3}, \eqref{ine33} and \eqref{ine333}.
\end{proof}

We state and prove the global well-posedness for semilinear equation \eqref{BouMatrixEq} in the following theorem.
\begin{theorem}\label{PeriodicThm}(Global in time mild solution)
Let $({\bf M},g)$ be a $d$-dimensional real hyperbolic manifold with $d\geq 2$. Let $1<p\leq d$ and $0<\delta<1$. 
Assume that the field $h$ satisfies Assumption \ref{Assum} and the external force $(F,f)\in \mathcal{Y}$. If the norms $\norm{Z_0}_{\mathbb{L}^p\cap \mathbb{L}^d}$, $\norm{h}_{\infty,L^{\frac{d}{\delta}}}$, $\norm{(F,f)}_{\mathcal{Y}}$  are sufficiently small, then equation \eqref{BouMatrixEq} has one and only one bounded mild solution $\hat{Z}(\cdot)=(\hat{u}(\cdot),\hat{\theta}(\cdot))$ in a small ball of $\mathcal X$.
\end{theorem} 
\begin{proof}
We denote the ball centered at $(0,0)$ with radius $\rho>0$ by 
\begin{eqnarray*}\label{bro}
\B_\rho &=&\{Z\in \mathcal X \hbox{ such that }   \norm{Z}_{\mathcal X}\le \rho \}.
\end{eqnarray*}
For a given function $W=(v,\eta)\in \B_\rho$, we consider the linear equation 
\begin{equation}\label{ns1}
Z(t)
=e^{-t\A}Z_0
+B\left(W,W\right)  (t)+T_h(\eta)(t) + \mathbb{T} \left(
\begin{bmatrix}
F\\
f
\end{bmatrix}
\right)(t).
\end{equation}
By applying the bilinear estimates in Lemma \ref{BiE} and the linear estimates in Lemma \ref{LinearEst1} and \ref{LinearEst2}, we obtain that: there exists a unique bounded mild solution $Z=(u,\theta)$ to \eqref{ns1} satisfying
\begin{eqnarray}\label{ephi}
\norm{Z}_{\mathcal X}&\le & C\norm{ Z_0}_{\mathbb{L}^p\cap \mathbb{L}^d} + K\norm{Z}^2_{\mathcal X}+ N \norm{h}_{\infty,L^{\frac{d}{\delta}}}\norm{(0,\eta)}_{\mathcal{X}}+  M\norm{(F,f)}_{\mathcal{Y}}\cr
&\le & C\norm{Z_0}_{\mathbb{L}^p\cap \mathbb{L}^d} + K\rho^2+ N \rho\norm{h}_{\infty,L^{\frac{d}{\delta}}} +  M\norm{(F,f)}_{\mathcal{Y}}\cr
&\leqslant& \rho
\end{eqnarray}
provided that $\rho$, $\norm{Z_0}_{\mathbb{L}^p\cap \mathbb{L}^d}$, $\norm{h}_{\infty,L^{\frac{d}{\delta}}}$ and $\norm{(F,f)}_{\mathcal{Y}}$ are small enough. Therefore, for these values we can define a map $\Phi: \mathcal{B}_\rho \to  \mathcal{B}_\rho$ as follows: $\Phi (W)(t)= Z(t)$
for all $t>0$, where $Z(\cdot)$ is a unique solution of linear equation \eqref{ns1}. Now, we prove that the mapping $\Phi$ is a contraction. Indeed, we have clearly
\begin{equation}\label{defphi1}
\Phi(W)(t) = Z_0 + B(W,W)(t) + T_h(\eta)(t) + \mathbb{T} \left( \begin{bmatrix}
    F\\ f
\end{bmatrix} \right)(t).
\end{equation}
Therefore, for all functions  $W_1=(u_1,\eta_1),\, W_2=(u_2,\eta_2)\in \B_\rho$, we apply again the bilinear estimates in Lemma \ref{BiE} 
and the linear estimate for $T(\cdot)$ in Lemma \ref{LinearEst1} to get 
\begin{eqnarray}\label{ephi33}
\norm{\Phi(W_1)-\Phi(W_2)}_{\mathcal X}
&=& \norm{ B(W_1,W_1) - B(W_2,W_2) + T_h(\eta_1-\eta_2)}_{\mathcal X}\cr
&\leqslant& \norm{B\left( \begin{bmatrix}u_1\\\eta_1\end{bmatrix},\begin{bmatrix}u_1\\\eta_1 \end{bmatrix}  \right) - B\left( \begin{bmatrix}u_2\\\eta_2\end{bmatrix},\begin{bmatrix}u_2\\\eta_2 \end{bmatrix} \right) + T_h(\eta_1-\eta_2)}_{\mathcal{X}}\cr
&\leq&\sup\limits_{t >0} \norm{ \int_0^t e^{-(t-\tau)\A}\dive \begin{bmatrix}
    u_2(u_2-u_1) + u_1(u_2-u_1)\\u_2(\eta_1-\eta_2) - \eta_1(u_2-u_1)
\end{bmatrix}(\tau)d\tau}^\blacklozenge   +\norm{  T_h(\eta_1-\eta_2)}_{\mathcal X}\cr
&\leq& \norm{ \begin{bmatrix}
    u_1-u_2\\\eta_1-\eta_2 \end{bmatrix}}_{\mathcal X} \left( \widehat{K}\norm{ \begin{bmatrix}
    u_1\\\eta_1 \end{bmatrix}}_{\mathcal X} + \widehat{K}\norm{ \begin{bmatrix}
    u_2\\ \eta_2 \end{bmatrix}}_{\mathcal X} + N \norm{h}_{\infty,L^{\frac{d}{\delta}}}\right)\cr
 &\leq&  \norm{W_1-W_2}_{\mathcal X} \left( 2\widehat{K}\rho + N \norm{h}_{\infty,L^{\frac{d}{\delta}}}\right).   
\end{eqnarray}
Therefore, for $\rho$ and $\norm{h}_{\infty,L^{\frac{d}{\delta}}}$ are small enough such that $\rho<\frac{1-N\norm{h}_{\infty,L^{\frac{d}{\delta}}}}{2\widehat{K}}$, we obtain that the mapping $\Phi$ is a contraction on $\B_\rho$.

Therefore, by fixed point arguments there exists a unique fixed point $\hat{W}=(\hat{u},\hat\theta)$ of $\Phi$, and by the definition of $\Phi$, this fixed point $\hat{W}$ is a bounded solution to semilinear equation \eqref{BouMatrixEq}.
The uniqueness of $\hat{W}$ in the small ball $\B_\rho$ holds as a direct consequence of \eqref{ephi33}. Our proof is complete.
\end{proof}

\section{Asymptotical behaviour and application}\label{S4}
\subsection{Exponential stability}
In this section, we prove the exponential stability of the mild solution to equation \eqref{BouMatrixEq} obtained in Theorem \ref{PeriodicThm}. Our main result is as follows: 
\begin{theorem}\label{stability}(Exponential stability).
Assume that the conditions of Theorem \ref{PeriodicThm} hold. Then, the mild solution $\hat{Z}=(\hat{u},\hat{\theta})$ of equation \eqref{BouMatrixEq} obtained in Theorem \ref{PeriodicThm} is exponentially stable in the sense that: for another  solution $\tilde{Z}=(\tilde{u},\tilde{\theta})\in \mathcal{X}$ of equation \eqref{BouMatrixEq} with the initial data $\tilde{Z}_0=(\tilde{u}_0,\tilde{\theta}_0)$ such that the norm $\norm{\hat{Z}_0-\tilde{Z}_0}_{\mathbb{L}^p\cap \mathbb{L}^d}$ is small enough, then we have 
	 \begin{eqnarray}\label{expstab}
	 	\norm{(\hat{Z}-\tilde{Z})(t)}^\blacklozenge	\lesssim  \norm{\hat{Z}_0-\tilde{Z}_0}_{\mathbb{L}^p\cap \mathbb{L}^d}  e^{-\bf\Theta t} \text{ for all } t>0,
	 \end{eqnarray} 
where ${\bf\Theta} = min\{\gamma_{p,p}, \beta_1,\beta_2, \hat{\beta}_1,\hat{\beta}_2,\tilde{\beta}_1-\alpha,\tilde{\beta}_2-\alpha \}$ with $\beta_i, \hat{\beta}_i, \tilde{\beta}_i,\; i={1,2}$ given in the previous sections.
\end{theorem} 
\begin{proof}
For $\hat{Z}=(\hat{u},\hat{\theta})$ and $\tilde{W}=(\tilde{u},\tilde{\theta})$ are two solutions of equation \eqref{BouMatrixEq} with initial data $(u_0,\theta_0)$ and $(\tilde{u}_0,\tilde{\theta}_0)$, respectively; we see that $(u-\tilde{u},\, \theta-\tilde{\theta})$ is solution of
\begin{eqnarray}\label{ephi3}
\hat{Z}(t)-\tilde{Z}(t)&=& e^{-t\A}(Z_0-\tilde{Z}_0)+ B\left(\begin{bmatrix}
				\hat{u}\\\hat{\theta}
		\end{bmatrix},\begin{bmatrix}
				\hat{u}\\\hat{\theta}
		\end{bmatrix} \right)(t) - B\left(\begin{bmatrix}
		\tilde{u}\\\tilde{\theta}
		\end{bmatrix},\begin{bmatrix}
		\tilde{u}\\\tilde{\theta}
		\end{bmatrix} \right)(t)\cr
		&&+ T_h(\hat\theta-\tilde{\theta}) \cr
&=&e^{-t\A}({Z}_0-\tilde{Z}_0)+  \int_0^t e^{-(t-\tau)\A}\dive \begin{bmatrix}
		\mathbb P[\tilde{u}(\tilde{u}-\hat{u}) + \hat{u}(\tilde{u}-\hat{u})]\\\tilde{u}(\hat\theta-\tilde{\theta})- \hat\theta(\tilde{u}-\hat{u})
		\end{bmatrix}(\tau)d\tau    \cr
		&&+  \int_0^t e^{-(t-\tau)\A} \begin{bmatrix}
			\mathbb P[h(\hat\theta-\tilde{\theta}) \\ 0
		\end{bmatrix}(\tau)d\tau.
	\end{eqnarray}
By the same manner as in the proof of Theorem \ref{PeriodicThm}, the equation \eqref{ephi3} has a unique mild solution $\hat{W}-\tilde{W}=(\hat{u}-\tilde{u},\, \hat{\theta}-\tilde{\theta})$ in a small ball $\mathcal{B}_{2\rho}$ of $\mathcal{X}$ provided that the norm
$\norm{Z_0-\tilde{Z}_0}_{\mathbb{L}^p\cap \mathbb{L}^d}$, $\norm{h}_{\infty,L^{d/\delta}}$ is small enough. In particular, we can chose $\norm{Z_0}_{\mathbb{L}^p\cap \mathbb{L}^d} \leq \rho,\, \norm{\tilde{Z}_0}_{\mathbb{L}^p\cap\mathbb{L}^d} \leq \rho$ for a given $\rho>0$. 

By the same way as in the proofs of Lemma \ref{LinearEst1} and Lemma \ref{BiE}, we can estimate that	
\begin{eqnarray}\label{Sta1}
\norm{\hat{Z}(t)-\tilde{Z}(t)}_{\mathbb{L}^p} &\leq&  \norm{e^{-t\A}(Z_0-\tilde{Z}_0)}_{\mathbb{L}^p}  +  \int_0^t\norm{ e^{-(t-\tau)\A}\dive \begin{bmatrix}
	\mathbb P[\tilde{u}(\tilde{u}-u) + \hat{u}(\tilde{u}-\hat u)]\\\tilde{u}(\hat\theta-\tilde{\theta})- \theta(\tilde{u}-\hat u)
	\end{bmatrix}(\tau)}_{\mathbb{L}^p}  d\tau    \cr
	&&+  \int_0^t \norm{e^{-(t-\tau)\A} \begin{bmatrix}
	\mathbb P[h(\hat{\theta}-\tilde{\theta}) \\ 0
	\end{bmatrix}(\tau)}_{\mathbb{L}^p} d\tau\cr
&\leq&  e^{-\gamma_{p,p}t}\norm{Z_0-\tilde{Z}_0}_{\mathbb{L}^p} + \int_0^t [h_d(t-\tau)]^{\frac{\delta+1}{d}}e^{-\beta_2(t-\tau)}\norm{\begin{bmatrix}
	\tilde{u}(\tilde{u}-\hat u) + \hat{u}(\tilde{u}-\hat{u})\\\tilde{u}(\hat \theta-\tilde{\theta})- \hat\theta(\tilde{u}-\hat u)
	\end{bmatrix}(\tau)}_{\mathbb{L}^{\frac{dp}{d+\delta p}}}d\tau\cr
	&&   + \int_0^t 			[h_d(t-\tau)]^{\frac{\delta}{d}}e^{-\beta_1(t-\tau)}\norm{\begin{bmatrix}
		h(\hat\theta-\tilde{\theta}) \\ 0
	\end{bmatrix}(\tau)}_{\mathbb{L}^{\frac{dp}{d+\delta p}}}d\tau\cr
&\leq&  e^{-\gamma_{p,p}t}\norm{Z_0-\tilde{Z}_0}_{\mathbb{L}^p} + \int_0^t [h_d(t-\tau)]^{\frac{\delta+1}{d}}e^{-\beta_2(t-\tau)}\norm{\hat{Z}-\tilde{Z}}_{\mathbb{L}^p}\left( \norm{\hat{Z}(\tau)}_{\mathbb{L}^{\frac{d}{\delta}}} +\norm{\tilde{Z}(\tau)}_{\mathbb{L}^{\frac{d}{\delta}}}\right) d\tau\cr
&&+ \int_0^t [h_d(t-\tau)]^{\frac{\delta}{d}}e^{-\beta_1(t-\tau)}\norm{h(\tau)}_{L^{\frac{d}{\delta}}}\norm{\hat{\theta}(\tau)-\tilde{\theta}(\tau)}_{L^p} d\tau\cr
&\leq&   e^{-\gamma_{p,p}t}\norm{Z_0 - \tilde{Z}_0}_{\mathbb{L}^p}  + \left( \norm{\hat{Z}}_{\mathcal X} +\norm{\tilde{Z}}_{\mathcal X}\right)\cr
&&\hspace{3.7cm}\times \int_0^t [h_d(t-\tau)]^{\frac{\delta+1}{d}}e^{-\beta_2(t-\tau)}[h_d(\tau)]^{\frac{1-\delta}{d}}e^{-\alpha \tau} \norm{\hat{Z}(\tau)-\tilde{Z}(\tau)}_{\mathbb{L}^p}   d\tau \cr
&&+ \norm{h}_{\infty,\frac{d}{\delta}}\int_0^t [h_d(t-\tau)]^{\frac{\delta}{d}}e^{-\beta_1(t-\tau)}\norm{\hat{Z}(\tau)-\tilde{Z}(\tau)}_{\mathbb{L}^p} d\tau,
\end{eqnarray}
where $\beta_1= \gamma_{dp/(1+\delta p),p},\; \beta_2=\dfrac{\gamma_{p,p}+\gamma_{dp/(1+\delta p),p}}{2}$.
 
Setting $y(\tau)=e^{\bf\Theta \tau} \norm{\hat{Z}(\tau)-\tilde{Z}(\tau)}_{\mathbb{L}^p} $ for ${\bf \Theta}<\min\{\gamma_{p,p},\beta_1 \}$. From \eqref{Sta1} and the fact that $\norm{\hat{Z}}_{\mathcal{X}}+\norm{\tilde{Z}}_{\mathcal{X}}\leq \rho+\rho=2\rho$, we obtain
\begin{eqnarray}\label{ine1}
y(t) &\leq& e^{-(\gamma_{p,p}-{\bf\Theta})t}\norm{Z_0 -\tilde{Z}_0}_{\mathbb{L}^p}  \cr
&&+ 2\rho\int_0^t [h_d(t-\tau)]^{\frac{\delta+1}{d}}e^{-(\beta_2-{\bf\Theta})(t-\tau)}[h_d(\tau)]^{\frac{1-\delta}{d}}e^{-\alpha \tau} y(\tau)   d\tau   
\cr
&&+ \norm{h}_{\infty,\frac{d}{\delta}}\int_0^t [h_d(t-\tau)]^{\frac{\delta}{d}}e^{-(\beta_1-{\bf\Theta})(t-\tau)}y(\tau) d\tau,
\end{eqnarray}
Since we have the following convergence of improper integrals (see Appendix \ref{app} for more details):
$$\int_0^t [h_d(t-\tau)]^{\frac{\delta+1}{d}}e^{-(\beta_2-{\bf\Theta})(t-\tau)}[h_d(\tau)]^{\frac{1-\delta}{d}}e^{-\alpha \tau}   d\tau\leq P<+\infty,$$
$$\int_0^t [h_d(t-\tau)]^{\frac{\delta}{d}}e^{-(\beta_1-{\bf\Theta})(t-\tau)} d\tau
= C^{\frac{\delta}{d}}\left[(\beta_1-{\bf\Theta})^{\frac{\delta}{2}-1}\Gamma\left( 1-\frac{\delta}{2}\right)  +\dfrac{1}{(\beta_1-{\bf\Theta})} \right]
\leq \tilde{P}<+\infty,$$
we can use Gronwall's inequality to get 
	$$y(t)\leq \norm{Z_0-\tilde{Z}_0}_{\mathbb{L}^p} e^{2\rho P+\tilde{P}}\; \text{ for all }t>0.$$
This leads to  
\begin{equation}\label{INE1}
\norm{\hat{Z}(t)-\tilde{Z}(t)}_{\mathbb{L}^p} \leqslant C e^{-{\bf\Theta} t} \norm{Z_0 - \tilde{Z}_0}_{\mathbb{L}^p} \text{ for all } t>0.
\end{equation}
By the same manner as above, we can also establish that 
\begin{equation}\label{INE2}
\norm{\hat{Z}(t)-\tilde{Z}(t)}_{\mathbb{L}^d} \leqslant C e^{-{\bf\Theta} t} \norm{Z_0 - \tilde{Z}_0}_{\mathbb{L}^d} \text{ for all } t>0.
\end{equation}

No, we estimate 
$z(\tau)=e^{{\bf\Theta} \tau}[h_d(\tau)]^{-\frac{1-\delta}{d}}e^{\alpha \tau}\norm{\hat{Z}(\tau)-\tilde{Z}(\tau)}_{\mathbb{L}^{\frac{d}{\delta}}}$. From the fact that $[h_d(t)]^{-\frac{1-\delta}{d}}\le C^{\frac{\delta-1}{d}}$ for all $t>0$, we have
\begin{eqnarray}\label{exp-stab}
	z(t)
&\leq& e^{{\bf\Theta} t}[h_d(t)]^{-\frac{1-\delta}{d}}e^{\alpha t} \norm{e^{-t\A}(Z_0-\tilde{Z}_0)}_{\mathbb{L}^{\frac{d}{\delta}}}\cr
&&+ e^{{\bf\Theta} t}[h_d(t)]^{-\frac{1-\delta}{d}}e^{\alpha t}  \int_0^t\norm{ e^{-(t-\tau)\A}\dive \begin{bmatrix}
		\mathbb P[\tilde{u}(\tilde{u}-\hat{u}) + \hat{u}(\tilde{u}-\hat{u})]\\\tilde{u}(\hat{\theta}-\tilde{\theta})- \hat{\theta}(\tilde{u}-\hat{u})
	\end{bmatrix}(\tau)}_{\mathbb{L}^{\frac{d}{\delta}}}d\tau  \cr
&& + e^{{\bf\Theta} t}[h_d(t)]^{-\frac{1-\delta}{d}}e^{\alpha t} \int_0^t \norm{e^{-(t-\tau)\A} \begin{bmatrix}
		\mathbb P[h(\hat{\theta}-\tilde{\theta}) \\ 0
\end{bmatrix}(\tau)}_{\mathbb{L}^{\frac{d}{\delta}}} d\tau \cr
&\leq& e^{({\bf\Theta}+\alpha-\gamma_{d,d/\delta}) t} \norm{ Z_0 - \tilde{Z}_0}_{\mathbb{L}^d}\cr
&&+ [h_d(t)]^{-\frac{1-\delta}{d}}e^{(\bf\Theta+\alpha) t}   \int_0^t [h_d(t-\tau)]^{\frac{\delta+1}{d}}e^{-\tilde{\beta}_2(t-\tau)}\norm{\begin{bmatrix}
			 \tilde{u}(\tilde{u}-\hat u) + \hat u(\tilde{u}-\hat u) \\\tilde{u}(\hat \theta-\tilde{\theta})- \hat \theta(\tilde{u}-\hat u)
		\end{bmatrix}(\tau)}_{\mathbb{L}^{\frac{d}{2\delta}}}d\tau  \cr
&&+ [h_d(t)]^{-\frac{1-\delta}{d}}e^{(\bf\Theta+\alpha) t} \int_0^t[h_d(t-\tau)]^{\frac{\delta}{d}}e^{-\tilde{\beta}_1(t-\tau)} \norm{ \begin{bmatrix}
		 h(\hat \theta-\tilde{\theta}) \\ 0
	\end{bmatrix}(\tau)}_{\mathbb{L}^{\frac{d}{2\delta}}} d\tau \cr
&\leq& e^{(\Theta+\alpha-\gamma_{d,d/\delta}) t} \norm{Z_0 - \tilde{Z}_0}_{\mathbb{L}^d}
+ [h_d(t)]^{-\frac{1-\delta}{d}}e^{(\bf\Theta+\alpha) t}   \cr&& \times\int_0^t [h_d(t-\tau)]^{\frac{\delta+1}{d}}e^{-\tilde{\beta}_2(t-\tau)}\left(\norm{\hat{Z}(\tau)}_{\mathbb{L}^{\frac{d}{\delta}}}+\norm{\tilde{Z}(\tau)}_{\mathbb{L}^{\frac{d}{\delta}}}\right)    \norm{(\hat{Z}-\tilde{Z})(\tau)}_{\mathbb{L}^{\frac{d}{\delta}}}d\tau \cr
&&+  [h_d(t)]^{-\frac{1-\delta}{d}}e^{(\bf\Theta+\alpha) t} \int_0^t[h_d(t-\tau)]^{\frac{\delta}{d}}e^{-\tilde{\beta}_1(t-\tau)} \norm{h(\tau)}_{L^{\frac{d}{\delta}}}\norm{\hat\theta(\tau)-\tilde{\theta}(\tau)}_{L^{\frac{d}{\delta}}} d\tau \cr
&\leq& e^{(\Theta+\alpha-\gamma_{d,d/\delta}) t} \norm{Z_0-\tilde{Z}_0}_{\mathbb{L}^d}	+ 2\rho C^{\frac{\delta-1}{d}} \int_0^t [h_d(t-\tau)]^{\frac{\delta+1}{d}}[h_d(\tau)]^{\frac{2(1-\delta)}{d}}e^{-({\tilde{\beta}_2}-\bf\Theta-\alpha)(t-\tau)} e^{-\alpha\tau}y(\tau)d\tau \cr
&&+  C^{\frac{\delta-1}{d}}\norm{h}_{\infty,L^{\frac{d}{\delta}}} \int_0^t[h_d(t-\tau)]^{\frac{\delta}{d}}e^{-(\tilde{\beta}_1-\bf\Theta-\alpha)(t-\tau)} y(\tau) d\tau,  
\end{eqnarray}
where $\tilde{\beta}_1= \gamma_{d/2\delta,d/\delta}$, $\tilde{\beta}_2= \frac{\gamma_{d/\delta,d/\delta} + \gamma_{d/2\delta,d/\delta}}{2}$.
Therefore, by using the fact that (see Appendix \ref{app}):
$$ \int_0^t [h_d(t-\tau)]^{\frac{\delta+1}{d}}[h_d(\tau)]^{\frac{2(1-\delta)}{d}}e^{-(\tilde{\beta}_2 -\bf\Theta-\alpha)(t-\tau)} e^{-\alpha\tau}d\tau\leq Q<+\infty,$$  
$$\int_0^t [h_d(t-\tau)]^{\frac{\delta+1}{d}}[h_d(\tau)]^{\frac{1-\delta}{d}}e^{-(\tilde{\beta}_1-\bf\Theta-\alpha)(t-\tau)} d\tau\leq \tilde{Q}<+\infty$$
and Gronwall's inequality, we can obtain from \eqref{exp-stab} that
$$z(t)\leq \norm{Z_0-\tilde{Z}_0}_{\mathbb{L}^d} e^{2\rho Q+\tilde{Q}}\; \text{ for all }t>0.$$
This is equivalent to
\begin{equation}\label{INE3}
[h_d(t)]^{-\frac{1-\delta}{d}}e^{\alpha t}\norm{\hat{Z}(t)-\tilde{Z}(t)}_{\frac{d}{\delta}} \leqslant C e^{-{\bf\Theta} t} \norm{Z_0 - \tilde{Z}_0}_{\mathbb{L}^d} \text{ for all } t>0.
\end{equation}
The exponential stability \eqref{expstab} holds by \eqref{INE1}, \eqref{INE2} and \eqref{INE3}. The proof is complete.
\end{proof}

\subsection{Application to the existence of periodic solutions}
In this part, we give an application of the stability obtained in Theorem \ref{stability} by establishing the existence of periodic mild solutions for Boussinesq system \eqref{BouMatrixEq}. The method is extended from \cite{HuyHa2022,HuyVan2022} and is called Serrin principle (see \cite{Se1959}).
\begin{theorem}\label{pest}
Assume that the conditions of Theorem \ref{PeriodicThm} hold.  
Then, if the functions $F$, $f$ and $h$ are periodic respecting on the time with periodicity $T$ (i.e., $T$-periodic functions), then equation \eqref{BouMatrixEq} has a  $T$-periodic solution $\hat{Z}=(\hat{u},\hat\theta)$ with the same periodicity $T$ in a small ball of $\mathcal X$. Moreover, the $T$-periodic solution to \eqref{BouMatrixEq} is locally unique in the sense that: Two $T$-periodic mild solutions $Z_1=(u,\theta)$ and $Z_2=(v,\xi)$ starting sufficiently near each other (i.e., $\|Z_1(0)-Z_2(0)\|_{\mathbb{L}^p\cap \mathbb{L}^d}$ is sufficiently small) are identical.  
\end{theorem} 
\begin{proof}
\def\xcal{\mathcal X}
\def\T{e^{-\A t}}
For each sufficiently small initial data $W=(x,y)\in \mathbb{L}^p({\bf M})\cap \mathbb{L}^d({\bf M})$, Theorem \ref{PeriodicThm}  follows that there exists a unique bounded mild solution $Z=(v,\xi)\in \mathcal X$ to equation
\eqref{BouMatrixEq} with $Z(0)=W$ in a small ball $\B_\rho$ of $\mathcal X$, if $\|(F,f)\|_{\mathcal{Y}}$ and $\norm{h}_{\infty,L^{\frac{d}{\delta}}}$ are small enough. More precisely, the facts that 
$$\rho<\min\left\{\frac{1}{2{K}},\frac{7}{16\widehat{K}}\right\},\; \norm{W}_{\mathbb{L}^p\cap \mathbb{L}^d} < \frac{\rho}{4C},\;\norm{h}_{\infty,L^{\frac{d}{\delta}}}<\frac{1}{8N}, \; \norm{(F,f)}_{\mathcal{Y}} < \frac{\rho}{8M}$$
guarantee the existence and uniqueness of such $Z$. 

Moreover, we can take an even smaller initial vector field $Z_0 = (u_0,\theta_0)$ such that 
$$\rho<\min\left\{\frac{1}{8K},\dfrac{7}{48\widehat{K}}\right\},\;\|Z_0\|_{\mathbb{L}^p\cap \mathbb{L}^d}\le \dfrac{\rho}{16C},\,\norm{h}_{\infty,L^{\frac{d}{\delta}}}<\dfrac{1}{32N}, \; \norm{(F,f)}_{\mathcal{Y}} < \dfrac{\rho}{32M}$$ 
(actually, $Z_0$ may be taken to be $(0,\,0)$). This leads to the fact that $Z=(u,\theta)$  $\in \B_{\frac{\rho}{4}}$, where $Z=(u,\theta)$ is the unique bounded mild solution to equation \eqref{BouMatrixEq}. That means $\norm{Z(t)}^\blacklozenge\le \dfrac{\rho}{4}$ for all $t>0$.

\def\n{\mathbb{N}} 
We now need to point out that the sequence  $\{Z(nT)=({u}(nT),\theta(nT))\}_{n\in \n}$   
is a Cauchy sequence in the space $\mathbb{L}^p({\bf M})\cap \mathbb{L}^d({\bf M})\cap \mathbb{L}^{\frac{d}{\delta}}({\bf M}) $ with the norm 
$$ \norm{Z(t)}^\blacklozenge= \norm{Z(t)}_{\mathbb{L}^{p}} + \norm{Z(t)}_{\mathbb{L}^d} + [h_d(t)]^{-\frac{1-\delta}{d}}e^{\alpha t} \norm{Z(t)}_{\mathbb{L}^{\frac{d}{\delta}}}.$$
Indeed, for arbitrary fixed natural  numbers  $m>n \in \n$, by putting $(z_1(t),z_2(t))=({u}(t+(m-n)T),\theta(t+(m-n)T)$,  
and using the periodicity of $F,f$ and $h$, we are easy to see that $z=(z_1,z_2)$ is also a mild solution to equation\eqref{BouMatrixEq}. Of course,  $z=(z_1,z_2)\in \B_{\rho/4}$. Therefore, Theorem \ref{stability}  implies that
\begin{equation}
\norm{Z(t)-z(t)}^\blacklozenge \leq C \norm{Z_0 - z(0)}_{\mathbb{L}^p\cap \mathbb{L}^d}  e^{-\bf\Theta t} \leq  K_0 e^{-\bf\Theta t},
\end{equation}
for all $t \geqslant 0$, where the constant $K_0$ independent of $m, n$.

Thence, by  $t:=nT$ in the above inequality and noting that $z(t)=(u(t+(m-n)T),\theta(t+(m-n)T)$, we imply that
\begin{equation}
\norm{Z(nT)-Z(mT)}^\blacklozenge \leqslant K_0 e^{-{\bf\Theta} (nT)},
\end{equation}
for all $m>n \in \mathbb{N}$.

This follows that $\{Z(nT)\}_{n\in \n}\subset \mathbb{L}^p(\bf M) \cap \mathbb{L}^d(\bf M) \cap \mathbb{L}^{d/\delta})(\mathbf M)$ is a Cauchy sequence.
Thus, the sequence $\{Z(nT)\}_{n\in \n}$ is convergent in $ \mathbb{L}^p(\bf M) \cap \mathbb{L}^d(\bf M) \cap \mathbb{L}^{d/\delta})(\mathbf M)  $ with 
$$\norm{Z(nT)}^\blacklozenge\le \dfrac{\rho}{4},$$ 
and we then put $Z^*= \lim_{n\to \infty}Z(nT).$
Clearly, we have the boundedness $\norm{Z^*}_{\mathbb{L}^p\cap\mathbb{L}^d}\le \dfrac{\rho}{4}$.

Taking now $Z^*$ as the initial data, by Theorem \ref{PeriodicThm}, we obtain that there exists a unique bounded mild solution $\hat{Z}(\cdot)=(\hat{u},\hat\theta)(\cdot)$ of the equation \eqref{BouMatrixEq} in the ball $\B_\rho$. 
We then prove that the mild solution $\hat{Z}(\cdot)$ is $T$-periodic. 
To do this, for each fixed $n\in\n$ we put $\omega(t)=(v(t),\xi(t))=(u(t+nT),\theta(t+nT))$ for $t\geq 0$. Again, by the periodicity of $F$, $f$ and $h$ we have
that $\omega(\cdot)$ is also a mild solution of equation\eqref{BouMatrixEq} with $\omega(0)=(u(nT),\theta(nT))$.

Since inequality \eqref{expstab} with $\omega$ instead of $Z=(u,\theta)$, we have
\begin{equation}
\norm{\hat{Z}(T)-\omega(T)}^\blacklozenge \lesssim \norm{\hat{Z}(0)-\omega(0)}_{\mathbb{L}^p\cap\mathbb{L}^d}  e^{-{\bf\Theta} T}.
\end{equation}
This means that
\begin{equation}
\norm{\hat{Z}(T)-Z((n+1)T)}^\blacklozenge \lesssim \norm{Z^* - Z(nT)}_{\mathbb{L}^p\cap\mathbb{L}^d}   e^{-{\bf\Theta} T}.
\end{equation}
Taking now $n\to \infty$ and utilizing the fact that
\begin{equation*}  
\lim_{n\to \infty}Z(nT)=Z^*=\hat{Z}(0)\in \mathbb{L}^p(\bf M)\cap \mathbb{L}^d(\bf M)\cap L^{\frac{d}{\delta}}(\bf M),
\end{equation*}
we obtain $\hat{Z}(T) = \hat{Z}(0).$
Consequently,  $\hat{Z}(\cdot)$ is  $T$-periodic. 

\noindent
The uniqueness of the $T$-periodic solution  follows from inequality \eqref{expstab}. 
Namely, if $Z=(u,\theta)$ and $z=(v,\xi)$ are two $T$-periodic mild solutions to equation \eqref{BouMatrixEq} with initial values  $Z_0=(u_0,\theta_0)$ and $z_0=(v_0,\xi_0)$ with $\|Z_0-z_0\|_{\mathbb{L}^p\cap \mathbb{L}^d}$ small enough, then inequality
 \eqref{expstab} implies that 
\begin{equation} 
 \lim_{t\to \infty}\norm{Z(t)-z(t)}^\blacklozenge =0.
 \end{equation}
Due to periodicity and continuity of $Z(\cdot)$ and $z(\cdot)$, this then yields that $Z(t)=z(t)$ for all $t\in \r_+$. Our proof is complete.
\end{proof}

\section{Appendix}\label{app}
\subsection{Generalized gravitational fields and well-posedness in three dimension case revisited}\label{B}
In this subsection, we verify that the gravitational field on the hyperbolic space ${\bf M}=\mathbb{H}^d$ with dimension $d\geqslant 3$ satisfying Assumption \ref{Assum}. First, we recall the formula of gravitational field in the case $d\geqslant 3$ as (in detail, see \cite[Section 2]{Barrow2020}):
\begin{equation}\label{gra}
\tilde{h}(r) = \frac{d\Phi(r)}{d r} = 
-\frac{(d-2)GM\coth^{d-3}r}{\sinh^2 r},
\end{equation}
where $r$ is the radius of a geodesic ball centered at the origin $O=(1,0,0...0)$ in hyperbolic manifold ${\bf M}$, $G$ is the gravitational constant and $\Phi(r) = GM\coth^{d-2}(r)$ is the gravitational potential acting on the test mass $M$. Here, we choose the curvature radius $R=1$ (in \cite{Barrow2020}, if $R\neq 1$, then $r$ is replaced by $\dfrac{r}{R}$ in equation \eqref{gra}). Observe that, for $d\geqslant 3$, we have the following equivalence
\begin{equation}\label{gra1}
\tilde{h}(r) \simeq
\begin{cases}
-GMe^{-2r}, \hbox{  as  } r\to \infty,\cr
-GMr^{-(d-1)}, \hbox{  as } r\to 0.
\end{cases}
\end{equation}
We observe that the gravitational field $\tilde{h}$ given by \eqref{gra} does not depend on time, then Assumption \ref{Assum} reduces to $\tilde{h}\in L^\infty(\Gamma(T{\bf M})) \cap L^{\frac{d}{2},\infty}(\Gamma(T{\bf M}))$. This condition is not valid on the whole space ${\bf M}$, but it can be valid on an exterior domain\footnote{If we consider the Boussinesq system \eqref{BouEq} on the exterior domain $\Omega$ in hyperbolic space ${\bf M}$, we need boundary conditions $u(\cdot,t)|_{\partial\Omega}=0$ and $\theta(\cdot,t)|_{\partial\Omega}=\omega(\cdot,t)$. This condition is similar to the one on an exterior domain in Euclid space (see \cite{Hi1997}).} $\Omega = {\bf M}-\mathbb{B}(O,\varepsilon)$, where  $\mathbb{B}(O,\varepsilon)$ is a geodesic ball in ${\bf M}$ centered at the origin $O= (1,0,0...0)$ with geodesic radius $\varepsilon$. In this context, we have $\tilde{h}\in L^\infty(\Gamma(T\Omega)) \cap L^{\frac{d}{2},\infty}(\Gamma(T\Omega))$. This condition is similar to the one given by Hishida (see conditions (3.1) and (3.2), page 61 in \cite{Hi1997}) when he considered the Boussinesq equation on exterior domain in Euclid space $\mathbb{R}^3$. 

Now, we verify $\tilde{h}\in L^\infty(\Gamma(T\Omega)) \cap L^{\frac{d}{2},\infty}(\Gamma(T\Omega))$, for $h$ given by \eqref{gra}. Clearly, on the exterior domain $\Omega$, the condition $\tilde{h} \in L^{\infty}(\Gamma(T\Omega))$ is valid. Moreover, the norm of interpolation space defined on hyperbolic manifold ${\bf M}$ is given by (see notions in the proof of Corollary 3.3 in \cite{An2009}):
\begin{equation}\label{LorentzNorm}
\norm{f}_{L^{q,\infty}} = \sup_{0<r<1}r^{\frac{d}{q}}|f(r)| + \sup_{r\geqslant 1} e^{\frac{d-1}{q}r}|f(r)|.
\end{equation}
Using equivalence \eqref{gra1}, we can show that the gravitational field $\tilde{h}$ on ${\bf M}$ (with dimension $d\geqslant 3$) satisfies $\norm{\tilde{h}}_{L^{\frac{d}{2},\infty}}<+\infty$, hence $\tilde{h}$ belongs to $L^{\frac{d}{2},\infty}(\Gamma(T\Omega))$.

In order to study the Boussinesq system \eqref{BouEq} on the whole space ${\bf M}$, we extend to consider a generalized gravitational field $h:{\bf M}\times \mathbb{R}_+\to \Gamma(T{\bf M})$ given by $h(x,t)= \alpha(x,t)\tilde{h}(x)$, where $\alpha:{\bf M}\times \mathbb{R}_+\to \mathbb{R}$, is a bounded and continuous function and has support outside the geodesic ball $\mathbb{B}(O,\varepsilon)$. Since $\tilde{h} \in L^\infty(\Gamma(T\Omega)) \cap L^{\frac{d}{2},\infty}(\Gamma(T\Omega))$ and the properties of $\alpha(x,t)$, the generalized gravitational field $h(x,t)=\alpha(x,t)\tilde{h}(x)$ satisfies Assumption \ref{Assum}, i.e., $h\in C_b(\r_+,L^\infty(\Gamma(T{\bf M}))\cap L^{\frac{d}{2},\infty}(\Gamma(T{\bf M})))$.
\begin{remark}
In the three dimension case, i.e., ${\bf M}=\mathbb{H}^3$, we can consider the field $h=\tilde{h}$ and obtain the well-posedness for Boussinesq systems on the whole space $\mathbb{H}^3$ by using the weak-$L^p$ spaces. In particular, from \eqref{gra1} and \eqref{LorentzNorm} we can verify directly that $\tilde{h}\in L^{\frac{3}{2},\infty}(\Gamma(T\mathbb{H}^3))$ which is coincidence to the property of gravitational field on Euclidean space $\mathbb{R}^3$. Therefore, by using $L^p-L^q$-dispersive estimates in Lemma \ref{estimates}, this property and Yamazaki's estimates we can process the same manner as in the previous works \cite{Fe2006,Fe2010,HuyXuan2022,XuanNgoc2025} to obtain the global well-posedness of the mild solutions for Boussinesq systems in the space $C_b(\r_+,L_\sigma^{3,\infty}(\Gamma(T\mathbb{H}^3))\times L^{3,\infty}(\mathbb{H}^3))$. The result reads as follows:
\begin{theorem}\label{PeriodicThm'}(Global well-posedness on $\mathbb{H}^3$).
Let $(\mathbb{H}^3,g)$ be a $3$-dimensional real hyperbolic manifold. 
Assume that the field $h \in L^{\frac{3}{2},\infty}(\Gamma(T\mathbb{H}^3))$ and the external force $(F,f)\in C_b(\r_+,\mathcal{L}^{\frac{3}{2},\infty}(\mathbb{H}^3))$, where $\mathcal{L}^{\frac{3}{2},\infty}(\mathbb{H}^3) = L^{\frac{3}{2},\infty}_\sigma(\mathbb{H}^3;\Gamma(T\mathbb{H}^3\otimes T\mathbb{H}^3))\times L^{\frac{3}{2},\infty}(\mathbb{H}^3;\mathbb{R}\otimes \mathbb{R})$. If the norms $\norm{Z_0}_{\mathbb{L}^{3,\infty}}$, $\norm{h}_{L^{\frac{3}{2},\infty}}$, $\norm{(F,f)}_{\infty, \mathcal{L}^{\frac{3}{2},\infty}}$  are sufficiently small, then equation \eqref{BouMatrixEq} has one and only one bounded mild solution $\hat{Z}(\cdot)=(\hat{u}(\cdot),\hat{\theta}(\cdot))$ in a small ball of $C_b(\r_+,\mathbb{L}^{3,\infty}(\mathbb{H}^3))$, where $\mathbb{L}^{3,\infty}(\mathbb{H}^3)=L_\sigma^{3,\infty}(\Gamma(T\mathbb{H}^3))\times L^{3,\infty}(\mathbb{H}^3)$.
\end{theorem}
\end{remark}

\subsection{Some calculations}\label{A}
In this section, we prove the boundedness of the following improper integrals which were used in the previous sections:
\begin{eqnarray*}
\tilde{N}_1&=&C^{\frac{\delta-1}{d}}\int_0^t [h_d(t-\tau)]^{\frac{\delta}{d}}[h_d(\tau)]^{\frac{1-\delta}{d}}e^{-({\tilde{\beta}_1}-\alpha)(t-\tau)} d\tau<+\infty;\cr
	 \tilde{M}_1&=&C^{\frac{\delta-1}{d}}\int_0^t [h_d(t-\tau)]^{\frac{\delta+1}{d}}[h_d(\tau)]^{\frac{1-\delta}{d}}e^{-(\tilde{\beta}_2-\alpha)(t-\tau)} d\tau<+\infty;\cr
K_1&=&\int_0^t [h_d(t-\tau)]^{\frac{\delta+1}{d}}e^{-\beta_2(t-\tau)}[h_d(\tau)]^{\frac{1-\delta}{d}}e^{-\alpha \tau}d\tau<+\infty;\cr
	\hat{K}_1&=&\int_0^t [h_d(t-\tau)]^{\frac{1+\delta}{d}}e^{-\hat{\beta}_2(t-\tau)}[h_d(\tau)]^{\frac{1-\delta}{d}}e^{-\alpha \tau}d\tau<+\infty;\cr 
\tilde{K}_1&=&C^{\frac{\delta-1}{d}}\int_0^t [h_d(t-\tau)]^{\frac{\delta+1}{d}}[h_d(\tau)]^{\frac{2(1-\delta)}{d}}e^{-(\tilde{\beta}_2-\alpha)(t-\tau)} e^{-\alpha\tau}d\tau<+\infty,	
\end{eqnarray*}
and
\begin{eqnarray*}
&&\int_0^t [h_d(t-\tau)]^{\frac{\delta+1}{d}}e^{-(\beta_2-\Theta)(t-\tau)}[h_d(\tau)]^{\frac{1-\delta}{d}}e^{-\alpha \tau}   d\tau <+\infty;\cr
&&[h_d(t)]^{-\frac{1-\delta}{d}}  \int_0^t [h_d(t-\tau)]^{\frac{\delta+1}{d}}[h_d(\tau)]^{\frac{2(1-\delta)}{d}}e^{-(\tilde{\beta}_2 -\bf\Theta-\alpha)(t-\tau)} e^{-\alpha\tau}d\tau <+\infty; \cr
&&[h_d(t)]^{-\frac{1-\delta}{d}}\norm{h}_{\infty,\frac{d}{\delta}}\int_0^t [h_d(t-\tau)]^{\frac{\delta+1}{d}}[h_d(\tau)]^{\frac{1-\delta}{d}}e^{-(\tilde{\beta}_1-\bf\Theta-\alpha)(t-\tau)} d\tau <+\infty,
 \end{eqnarray*}
where
$$\beta_1= \gamma_{dp/(1+\delta p),p},\;\beta_2= \dfrac{\gamma_{p,p} + \gamma_{dp/(1+\delta p),p}}{2},\;\hat{\beta}_1= \gamma_{d/(1+\delta),d},$$
$$\hat{\beta}_2= \dfrac{\gamma_{d,d}+\gamma_{d/(1+\delta),d}}{2},\;\tilde{\beta}_1= \gamma_{d/2\delta,d/\delta},\;\tilde{\beta}_2= \dfrac{\gamma_{d/\delta,d/\delta} + \gamma_{d/2\delta,d/\delta}}{2}.$$ 

The boundedness of the integrals $\tilde{N}_1,\,\tilde{M}_1,\, K_1,\,\hat{K}_1,\, P$ and $\tilde{Q}$ are proved similarly. We prove only the boundedness of $\tilde{N}_1$. Indeed, we consider the following cases of $t$:

{\bf For the case: $0<t<1$.} It is clear that
\begin{eqnarray*}
	\tilde{N}_1 &=& C^{\frac{\delta-1}{d}}\int_0^t [h_d(t-\tau)]^{\frac{1+\delta}{d}} [h_d(\tau)]^{\frac{1-\delta}{d}} e^{-(\tilde{\beta}_1-\alpha) (t-\tau)} d\tau\cr
	&\leq& C^{\frac{\delta-1}{d}} \int_0^t (t-\tau)^{-\frac{1+\delta}{2}} \tau^{-\frac{1-\delta}{2}} d\tau\cr
	&\leq&C^{\frac{\delta-1}{d}} {\bf B}\left( \frac{1-\delta}{2},\,\frac{1+\delta}{2} \right)<+\infty,
\end{eqnarray*}
where ${\bf B}(\cdot,\, \cdot)$ is the beta function.

{\bf For the case: $1\leq t$}. We imply that
\begin{eqnarray*}
	\tilde{N}_1 &=& C^{\frac{\delta-1}{d}}\int_0^t [h_d(t-\tau)]^{\frac{1+\delta}{d}} [h_d(\tau)]^{\frac{1-\delta}{d}} e^{-(\tilde{\beta}_1-\alpha) (t-\tau)} d\tau\cr
	&=& C^{\frac{\delta-1}{d}} \int_0^1 [h_d(t-\tau)]^{\frac{1+\delta}{d}} [h_d(\tau)]^{\frac{1-\delta}{d}} e^{-(\tilde{\beta}_1-\alpha) (t-\tau)} d\tau \cr
	&&+C^{\frac{\delta-1}{d}} \int_1^t [h_d(t-\tau)]^{\frac{1+\delta}{d}} [h_d(\tau)]^{\frac{1-\delta}{d}} e^{-(\tilde{\beta}_1-\alpha) (t-\tau)} d\tau\cr
	&\leq & C^{\frac{\delta-1}{d}} \int_0^1 \left((t-\tau)^{-\frac{1+\delta}{2}}+1\right) \tau^{-\frac{1-\delta}{2}} e^{-(\tilde{\beta}_1 - \alpha) (t-\tau)} d\tau \cr
	&&+C^{\frac{\delta-1}{d}} \int_1^t (t-\tau)^{-\frac{1+\delta}{2}} e^{-(\tilde{\beta}_1-\alpha) (t-\tau)} d\tau\cr
	&\leq& C^{\frac{\delta-1}{d}} \int_0^t (t-\tau)^{-\frac{1+\delta}{2}}\tau^{-\frac{1-\delta}{2}} d\tau + C^{\frac{\delta-1}{d}}\int_0^1 \tau^{-\frac{1-\delta}{2}} d\tau\cr
	&&+C^{\frac{\delta-1}{d}} \int_0^t (t-\tau)^{-\frac{1+\delta}{2}} e^{-(\tilde{\beta}_1-\alpha) (t-\tau)} d\tau\cr
	&\leq &C^{\frac{\delta-1}{d}} \left[ {\bf B}\left( \frac{1-\delta}{2},\, \frac{1+\delta}{2} \right) + \frac{2}{1+\delta} + (\tilde{\beta}_1-\alpha)^{-\frac{1-\delta}{2}}\mathbf{\Gamma}\left(\frac{1-\delta}{2}\right)\right] <+\infty,
\end{eqnarray*}
where $\mathbf{\Gamma}(\cdot)$ is the gamma function.

The boundedness of integral $\tilde{K}_1$ and $Q$ are pointed out similarly, we prove only for $\tilde{K}_1$ as follows.

{\bf For the case: $0<t<1$}. We have
\begin{eqnarray*}
\tilde{K}_1 &=& C^{\frac{\delta-1}{d}} \int_0^t [h_d(t-\tau)]^{\frac{1+\delta}{d}} [h_d(\tau)]^{\frac{2(1-\delta)}{d}} e^{-(\tilde{\beta}_2-\alpha) (t-\tau)} e^{-\alpha \tau} d\tau\cr
	&\leqslant& C^{\frac{\delta-1}{d}}  \int_0^t (t-\tau)^{-\frac{1+\delta}{2}} \tau^{-(1-\delta)} d\tau\cr
	&\leqslant&C^{\frac{\delta-1}{d}}  t^{\frac{1-\delta}{2}}{\bf B}\left( \frac{1-\delta}{2},\,\delta \right)<+\infty,\; \hbox{ because }  0<t<1.
\end{eqnarray*}

{\bf For the case: $1\leq t$}. It is not hard to get following estimates.
\begin{eqnarray*}
	\tilde{K}_1 &=&C^{\frac{\delta-1}{d}}  \int_0^t [h_d(t-\tau)]^{\frac{1+\delta}{d}} [h_d(\tau)]^{\frac{2(1-\delta)}{d}} e^{-(\tilde{\beta}_2-\alpha) (t-\tau)} e^{-\alpha \tau} d\tau\cr
	&=&C^{\frac{\delta-1}{d}}   \int_0^1 [h_d(t-\tau)]^{\frac{1+\delta}{d}} [h_d(\tau)]^{\frac{2(1-\delta)}{d}} e^{-(\tilde{\beta}_2-\alpha) (t-\tau)} d\tau \cr
	&&+C^{\frac{\delta-1}{d}}  \int_1^t [h_d(t-\tau)]^{\frac{1+\delta}{d}} [h_d(\tau)]^{\frac{2(1-\delta)}{d}} e^{-(\tilde{\beta}_2-\alpha) (t-\tau)} d\tau\cr
	&\leqslant&C^{\frac{\delta-1}{d}}   \int_0^1 \left((t-\tau)^{-\frac{1+\delta}{2}}+1\right) \tau^{-(1-\delta)}e^{-(\tilde{\beta}_2-\alpha)\tau} d\tau \cr
	&&+C^{\frac{\delta-1}{d}}  \int_1^t (t-\tau)^{-\frac{1+\delta}{2}} e^{-(\tilde{\beta}_2-\alpha) (t-\tau)} d\tau\cr
	&\leqslant&C^{\frac{\delta-1}{d}} \left[   \left(2^{\frac{1+\delta}{2}}+1\right) \int_0^{1/2} \tau^{-(1-\delta)} d\tau + 2^{1-\delta}\int_{1/2}^1 \left((t-\tau)^{-\frac{1+\delta}{2}}+1\right) e^{-(\tilde{\beta}_2-\alpha)\tau}d\tau\right] \cr
	&&+C^{\frac{\delta-1}{d}}  \int_0^t (t-\tau)^{-\frac{1+\delta}{2}} e^{-(\tilde{\beta}_2-\alpha) (t-\tau)} d\tau\cr
	&\leqslant&C^{\frac{\delta-1}{d}} \left[  \left(2^{\frac{1+\delta}{2}}+1\right)\frac{1}{\delta 2^\delta} + \frac{2^{1-\delta}}{\tilde{\beta}_2-\alpha}\left( e^{-\frac{\tilde{\beta}_2-\alpha}{2}}- e^{-(\tilde{\beta}_2-\alpha)}\right)\right]  \cr
	&&+C^{\frac{\delta-1}{d}}  (2^{1-\delta}+1) (\tilde{\beta}_2-\alpha)^{-\frac{1-\delta}{2}}\mathbf{\Gamma}\left(\frac{1-\delta}{2}\right)<+\infty.
\end{eqnarray*}

\end{document}